\documentclass[final,leqno]{amsart}

\usepackage{graphicx}
\usepackage{amssymb}
\usepackage{amsmath}

\usepackage{hyperref}
\usepackage{mathrsfs}
\usepackage{float}
\usepackage{mathtools}

\usepackage[usenames]{xcolor}

\newtheorem{teo}{Theorem}[section]

\theoremstyle{definition}

\newtheorem{prop}[teo]{Proposition}

\newtheorem{assunp}[teo]{Assumption}
\newtheorem{hypo}{Hypothesis}
\newtheorem{problem}{\bfseries Problem\rmfamily}
\theoremstyle{remark}

\numberwithin{equation}{section}

\newcommand{\mV}{\mathcal{V}}
\newcommand{\mQ}{\mathcal{Q}}
\newcommand{\mK}{\mathcal{K}}
\newcommand{\mH}{\mathcal{H}}

\newcommand{\mF}{\mathcal{F}}
\newcommand{\mG}{\mathcal{G}}
\newcommand{\eu}{\texttt{e}_u}
\newcommand{\ep}{\texttt{e}_p}
\newcommand{\ew}{\texttt{e}_w}
\newcommand{\emm}{\texttt{e}_m}

\DeclareMathOperator{\Ima}{Im}
\newcommand\mA{\mathbb{A}}

\newcommand\mB{\mathbb{B}}

\newcommand\mL{\mathbb{L}}
\newcommand{\dive}{\operatorname{div}}

\usepackage{array}   
\newcolumntype{C}{>{$}c<{$}} 
\usepackage{threeparttable}
\newcolumntype{A}{>{\centering}p{0.1\textwidth}}
\newcolumntype{B}{>{\centering}p{1cm}}
\begin{document}

\title[Abstract mixed viscoelastic formulation]{Analysis of an abstract mixed formulation for viscoelastic problems}

\author{Erwin Hern\'andez}
\address{Departamento de Matem\'atica, Universidad T\'ecnica Federico Santa Mar\'ia,Casilla 110-V Valparaiso, Chile}
\email{erwin.hernandez@usm.cl}
\thanks{The first author has been partially supported by ANID-Chile through FONDECYT No. 1181098, Chile}

\author{Felipe Lepe}
\address{GIMNAP-Departamento de Matem\'atica, Universidad del B\'io-B\'io, Casilla 5-C, Concepci\'on, Chile}
\email{flepe@ubiobio.cl}
\thanks{The second author has been partially supported by ANID-Chile through FONDECYT project No. 11200529, Chile}

\author{Jesus Vellojin}
\address{Departamento de Matem\'atica, Universidad T\'ecnica Federico Santa Mar\'ia,Casilla 110-V Valparaiso, Chile}
\email{jesus.vellojinm@usm.cl}
\thanks{The third author has been partially supported by ANID-Chile through FONDECYT No. 1181098, Chile}

\date{}

\dedicatory{}


\maketitle
\begin{abstract}
	This study provides an abstract framework to analyze mixed formulations in viscoelasticity, in the classic saddle point form. Standard hypothesis for mixed methods are adapted to the Volterra type equations in order to obtain stability of the proposed problem. Error estimates are derived for suitable finite element spaces. We apply the developed theory to a bending moment formulation for a linear viscoelastic Timoshenko beam and for the Laplace operator with memory terms. For both problems we report numerical results to asses the performance of the methods.
	
\end{abstract}

\section{Introduction}
\label{sec:intro}
\subsection{Scope}
The classic  mixed saddle point problems of the form: find $(\boldsymbol{v},p)\in\mathcal{V}\times\mathcal{Q} $ such that
\begin{align}
	a(\boldsymbol{v},\boldsymbol{\tau})+b(\boldsymbol{\tau},p)&=\langle f,\boldsymbol{\tau}\rangle_{\mV}\quad\forall\boldsymbol{\tau}\in\mathcal{V},\label{1}\\
	b(p,q)&=\langle g,q\rangle_\mQ\quad\forall q\in \mathcal{Q},\label{2}
\end{align}
where $\mathcal{V},\mathcal{Q}$ are Hilbert spaces, $\langle\cdot,\cdot\rangle_\mV$ is the duality pairing between $\mV$ and its dual $\mV'$, and $\langle\cdot,\cdot\rangle_\mQ$ is the corresponding duality pairing between $\mQ$ and $\mQ'$, has been fully studied in the past years, and the theory to analyze existence and uniqueness can be found in the classic book \cite{boffi2013mixed}.

In this paper, we consider the extension of \eqref{1}--\eqref{2} and incorporate the corresponding history terms, inherent for problems with memory, and develop an ad-hoc abstract theory to establish the existence and uniqueness of a system of Volterra equations arising from the Boltzmann principle. The goal problem of our paper is the following:

\begin{problem}\label{prob4}
	Find a pair $(u,p)\in L^1(\mathcal{J};\mV\times \mQ)$ such that
	\begin{equation*}
		\label{weak-formulation-1}
		\left\{\begin{aligned}
			a(u,v) + b(v,p) &= \langle f,v\rangle_{\mV} + \int_{0}^{t}\bigg[k_1(t,s) a(u(s),v) + k_2(t,s)b(v,p(s))\bigg]\,ds,\\
			b(u,q)&=\langle g,q\rangle_{\mQ} + \int_{0}^tk_3(t,s)b(u(s),q)ds,
		\end{aligned}\right.
	\end{equation*}
	for all $(v,q)\in\mV\times\mQ$,
\end{problem}
Here, $k_i$ with  $i=1,2,3$,  are bounded and continuous Volterra kernels. This type of system, with different kind of the Volterra kernels, are present in non-local flows in porous media, viscoelasticity, contact problems, among others (see, for instance, \cite{ewing2002sharp,sinha2009mixed,karaa2015optimal,matei2018mixed}).
We provide an analysis for the continuous formulation, together with its finite element discretization, and error analysis in suitable norms. Note that the proposed problem above gives flexibility on how to study several models arising from different phenomenas involving materials with memory. While mixed formulations deals with the spatial behavior, the theory of integro-differential equations deals with the implications in time, providing regularity results that proves to be useful when using numerical schemes.  

The analysis for our paper consists in the rigorous derivation of several estimates for mixed viscoelastic formulation like Problem \ref{prob4}, where the constants of each of these estimates are completely described in order to observe the dependency of other constants that may deteriorate the error estimates of some particular numerical method. An important example in this context are the slender thin structures, like 
beams, plates, and shells, since their response when external forces are applied depend on the thickness parameter that can produce numerical locking when standard numerical methods are used; this effect can be avoided when mixed formulations are considered. 

\subsection{Related work}
It is well known that viscoelastic properties for different materials are common in different applications, like industry, the design of different devices that depend on structures with viscoelastic responses, biomechanics, fluid, and solid mechanics etc. Due this applications, and others that may appear, it is important to analyze viscoelastic models and numerical methods to approximate the corresponding solutions. From the physical point of view, if we focus on thin structures like plates or beams, the presence of viscoelastic properties  leads to the damping effect that causes dissipation of energy during deformation of the structures.  We refer to Christensen \cite{christensen2012theory} or Flugge \cite{flugge1975viscoelasticity} for further details about the physical phenomenons involved. Similar effects can be analyzed with viscoelastic fluids, as plasma, magma or some types of non-Newtonian fluids, although their mathematical analysis is complex \cite{hanks2003fluid}.

The importance of the study of mixed formulations for viscoelastic problems lies in the fact that the incorporation of additional unknowns are important in some applications due to their physical meaning. For example, in \cite{sinha2009mixed,karaa2015optimal}, for the Laplace problem with memory, the additional unknown that is incorporated represents the velocity of some fluid in porous media, instead of the  classic stress in the steady problems. Research in the field of contact problems in viscoelasticity also make use of mixed formulation in order to introduce Lagrange multipliers to guarantee existence and uniqueness of a solution, which later helps to provide robust numerical schemes (see for example \cite{matei2018mixed} for a case of contact in viscoelasticity using a Kelvin-Voigt material).

On the other hand, it is well known that mixed formulations
avoid the locking phenomenon that arises due some parameter (or
parameters) present in the models \cite{boffi2013mixed}. This is the case of slender structures. For instance, there is the work of \cite{rognes2010}, where a weak symmetry formulation is considered that leads to a mixed formulation in linear viscoelasticity for solids, that is capable of avoid locking. The authors use constitutive equations in differential operator form, where the models are associated with Maxwell and Kelvin-Voigt materials. Recently, a new approach to deal with the numerical locking-free approximation of thin viscoelastic shells has been introduced in  \cite{hernandez2019}, where the authors had obtained a new viscoelastic shell formulation, based in the mixed formulation of pure elastic shell in \cite{chapelle2010finite}. There is also the corresponding principle, that allows to translate the viscoelastic problem into an equivalent elastic one, for which locking-free numerical methods can be applied. Although this approach is useful to compute exact solutions, it has several drawbacks since it depends on the time invariance of the boundary conditions and the nature of the relaxation modulus \cite{mukherjee2003elastic}.

\subsection{Outline}
The paper is organized as follows: in Section \ref{sec:abstract_analysis} we introduce the framework in which we will analyze Problem \ref{prob4}. Results for the existence and uniqueness of the viscoelastic mixed problem are established. Also, we derive a series of estimates in order to provide several stability results that allows to obtain semi-discrete error estimates. Due to the abstract nature of our model, the resulting constants are given explicitly such that they can be adjusted according to the applications. In Section \ref{sec:applications_timoshenko_beam}, we present two examples in linear viscoelasticity. First, we propose a brief analysis on the Laplace equation with memory, widely used to model flows in porous media and non-fickian flows with memory. We show that a mixed formulation obtained from this model fits the requirements developed in Section \ref{sec:abstract_analysis}, thus stability and error estimates are obtained. On the other hand, we analyze a linear viscoelastic Timoshenko beam model, which is rewritten as a mixed system following the approach of \cite{lepe2014locking}. Using the coercivity and inf-sup conditions provided from the elastic case at the continuous and semi-discrete level, we obtain the respective approximation results and convergence rates, resulting in a locking-free finite element formulation. Both applications are supported with numerical experiments to asses the performance of the mixed methods.

\section{Abstract mixed formulation}
\label{sec:abstract_analysis}
\subsection{Preliminar assumptions}

Let $\mV, \mQ$ two  real Hilbert spaces endowed with norms $\Vert \cdot\Vert_{\mV}$ and $\Vert \cdot\Vert_{\mQ}$, respectively.  We denote by $\mathcal{L}(\mV;\mQ)$ the space of linear and bounded operators defined between $\mV$ and $\mQ$. Let $\mV'$ and $\mQ'$ be the corresponding dual spaces of $\mV$ and $\mQ$, respectively, endowed with norms $\Vert\cdot\Vert_{\mV'}$ and $\Vert \cdot\Vert_{\mQ'}$. Let $\mathcal{J}:=[0,T],\, T\in [0,\infty[$. From the boundedness of $k_i,\, i=1,2,3,$ we assume that there exists $C_{k_i}\geq 0$ such that 
\begin{equation*}
	\vert k_i(t,s)\vert\leq C_{k_i},
\end{equation*}
Also, we assume that each $k_i$ belongs almost everywhere to the triangle (see for instance \cite{gutierrez2014engineering}) 
\begin{equation*}
	\mathcal{T}:= \bigg\{ \tau \in \mathcal{J} \;\;\vert\;\; 0\leq \tau\leq t,\quad t \in \mathcal{J} \bigg\}.
\end{equation*}

For the sake of simplicity, for every Banach space $\mathcal{B}$ and every time interval $[0,t]$, we will denote by $L_{[0,t]}^\ell(\mathcal{B})$ the Bochner space $L^{\ell}(\mathcal{J};\mathcal{B})$, endowed with norm
$$
\Vert \mathfrak{w}\Vert_{L_{[0,t]}^\ell(\mathcal{B})}:=\left(\int_{0}^t\Vert \mathfrak{w}\Vert_{\mathcal{B}}^p\right)^{1/\ell},
$$
for $1\leq \ell <\infty,$ with the usual modification for $\ell=\infty$. If $\mathcal{S}:=[0,t]$ for any $t\in\mathbb{R}^+$, we will simply write $L_{\mathcal{S}}^\ell(\mathcal{B})$.

Let us introduce the following general assumptions  (see \cite{boffi2013mixed} for details).
\begin{assunp}
	\label{assumption-1}
	Let $\mV$ and $\mQ$ be two Hilbert spaces.  Let $a(\cdot,\cdot):\mV\times \mV\rightarrow \mathbb{R}$ and $b(\cdot,\cdot):\mV\times \mQ\rightarrow\mathbb{R}$ be two given bilinear forms, and denote by $\mA:\mV\rightarrow\mV$, $\mB:\mV\rightarrow\mQ$ their corresponding induced linear operators. Assume that the following properties are satisfied:
	\begin{itemize}
		
		\item [i.] The bilinear form $a(\cdot,\cdot)$ is symmetric, positive semi-definite and continuous on $\mV$, respectively, i.e., 
		\begin{equation*}
			\begin{aligned}
				&\vert a(v,w)\vert \leq \Vert \mA \Vert_{\mathcal{L}(\mV;\mV')} \Vert v\Vert_\mV\Vert w\Vert_\mV \equiv\Vert a \Vert_{\mathcal{L(V\times\mV;\mathbb{R})}}\Vert v\Vert_\mV\Vert w\Vert_\mV, \\
			\end{aligned}
		\end{equation*}
		for all $w,v\in\mV$. We will also require that $a(\cdot,\cdot)$ be strongly coercive in $\mathcal{K}$, respectively, i.e., there exist $\alpha_0>0$ such that
		\begin{equation*}
			\begin{aligned}
				&a(v_0,v_0)\geq \alpha_0\Vert v_0\Vert_\mV^2 \hspace{0.4cm} \forall v_0\in \mK.
			\end{aligned}
		\end{equation*}
		The norms of the induced operator and bilinear form will be denoted simply by $\Vert a \Vert$. For this,  we have that
		\begin{equation*}
			\begin{aligned}
				&\langle \mA w,v\rangle_{\mV}=\langle w,\mA v\rangle_{\mV}=a(w,v)\hspace{0.4cm}\forall w,v\in\mV.
			\end{aligned}
		\end{equation*}
		\item [ii.]  Similarly, the bilinear form $b(\cdot, \cdot)$ is continuous, i.e.,  
		\begin{equation*}
			\vert b(v,q)\vert\leq \Vert \mB\Vert_{\mathcal{L}(\mV;\mQ')} \Vert v\Vert_\mV\Vert q\Vert_\mQ= \Vert b\Vert_{\mathcal{L(V\times}\mQ;\mathbb{R})}\Vert v\Vert_\mV\Vert q\Vert_\mQ,
		\end{equation*}
		and similarly, the norms of the operator and bilinear form will be denoted by $\Vert b\Vert$. Moreover, the  operators $\mB$ satisfies
		\begin{equation*}
			\langle \mB v,q\rangle_{\mQ}=\langle v, \mB^*q\rangle_{\mV}=b(v,q) \hspace{0.4cm}\forall v\in \mV, \forall q\in\mQ.
		\end{equation*}
		\item [iii. ] We consider functions $f$ and $g$ that are continuous on $\mV$ and $\mQ$, respectively, i.e.,
		$$
		\langle f,v\rangle_{\mV}\leq \Vert f \Vert_{\mV'}\Vert v\Vert_\mV \hspace{0.2cm}\forall v\in \mV,
		$$		and 
		$$
		\langle g,q\rangle_{\mQ}\leq \Vert g \Vert_{\mQ'}\Vert q\Vert_\mQ\hspace{0.2cm}\forall q\in \mQ.
		$$
		a.e. in $\mathcal{J}$.
	\end{itemize}
\end{assunp}

We recall the following  classic result of mixed formulations  (See \cite{boffi2013mixed} for instance).

\begin{teo}
	\label{teo-1}
	Let $\mV$ and $\mQ$ be two real Hilbert spaces and $\mB$ a linear continuous operator from $\mV$ to $\mQ'$.
	Then, the following statements are equivalent:
	\begin{enumerate}
		\item $\Ima \mB=\mQ'$.
		\item  $ \Ima \mB^*$ is closed and $\mB^*$ is injective.
		\item There exists $\beta>0$ such that $\Vert \mB^*q\Vert_{\mV'}\geq \beta \Vert q\Vert_\mQ, \forall q\in \mQ$.
		\item There exists a lifting operator $\mL_{\mB}\in \mathcal{L}(\mQ';\mV)$ such that $\mB(\mL_{\mB}(g))=g$, for all $g\in \mQ'$, with $\Vert \mL_{\mB}\Vert\leq1/\beta$.
	\end{enumerate}
\end{teo} 

We remark that Theorem \ref{teo-1} is valid a.e. in $\mathcal{J}$. From now on, we will omit the time and space dependencies, unless they are necessary in the proofs.

\subsection{Mixed formulation analysis}


From Assumption \ref{assumption-1} we have the corresponding operators form of Problem \ref{prob4}.
\begin{problem}\label{prob3}
	Given $f\in L_\mathcal{J}^1(\mV')$ and $g\in L_\mathcal{J}^1(\mQ')$, find $(u,p)\in L_\mathcal{J}^1(\mV\times\mQ)$ such that
	\begin{equation*}
		\label{linear-mixed-problem1}
		\left\{\begin{aligned}
			\mA u(t) +\mB^*p(t) &=f(t) + \int_{0}^{t}\bigg[k_1(t,s)\mA u(s)+ k_2(t,s)\mB^*p(s)\bigg]ds,\\
			\mB u(t)&= g(t) +\int_{0}^tk_3(t,s)\mB u(s)ds.
		\end{aligned}\right.
	\end{equation*}
\end{problem}

The following result establish the existence and uniqueness of a solution to Problem \ref{prob4}.
\begin{teo}
	\label{teo-2}
	Under the Assumption \ref{assumption-1}, assume that $\Ima \mB=\mQ'$. Then, Problem \ref{prob4} has a unique solution for all $f\in L_{\mathcal{J}}^1(\mV')$ and $g\in L_{\mathcal{J}}^1(\mQ')$.
\end{teo}
\begin{proof}
	Let us write Problem \ref{prob3} in the following form:
	\begin{equation}
		\label{teo-2-001}
		\displaystyle\begin{bmatrix}
			\mA & \mB^*\\
			\mB & 0
		\end{bmatrix}\begin{bmatrix}
			u\\
			p
		\end{bmatrix}
		=
		\begin{bmatrix}
			f\\
			g
		\end{bmatrix}
		+
		\int_{0}^t\begin{bmatrix}
			k_1(t,s)\mA & k_2(t,s)\mB^*\\
			k_3(t,s)\mB & 0
		\end{bmatrix}\begin{bmatrix}
			u(s)\\
			p(s)
		\end{bmatrix}ds.
	\end{equation}
	
	We begin by proving the invertibility of the following operator
	\begin{equation}
		\label{teo-2-matrix}
		\begin{bmatrix}
			\mA & \mB^*\\
			\mB & 0
		\end{bmatrix}.
	\end{equation}
	To do this task, let us consider the equality
	\begin{equation}
		\label{teo-2-002}
		\begin{bmatrix}
			\mA & \mB^*\\
			\mB & 0
		\end{bmatrix}\begin{bmatrix}
			\hat{u}\\
			\hat{p}
		\end{bmatrix}
		=
		\begin{bmatrix}
			0\\
			0
		\end{bmatrix},
	\end{equation}
	where it is clear that  $\mB \hat{u}=0$ and hence $\hat{u}\in\mK$. Also, from the first equation in \eqref{teo-2-002} we have that
	\begin{equation}
		\label{teo-2-003}
		\mA \hat{u} + \mB^*\hat{p}=0.
	\end{equation}
	
	From the $\mK-$ellipticity of $a(\cdot,\cdot)$, we conclude that the restriction to $\mK$ of the associated operator $\mA$, denoted by $\mA_{\mK\mK'}$, is an isomorphism that satisfies
	$$
	\langle A_{\mK\mK'}\hat{u}_0,v_0\rangle_{\mK'\times\mK}=\langle \hat{u}_0,A_{\mK\mK'}v_0\rangle_{\mK'\times\mK}=a(\hat{u}_0,v_0)\qquad\forall \hat{u}_0,v_0\in\mK.
	$$
	
	On the other hand, for all $\hat{u}_0\in \mK$, we have that
	$$
	\langle \mB^*\hat{p},\hat{u}_0\rangle_{\mV}=\langle \hat{p},\mB \hat{u}_0\rangle_{\mQ}=0.
	$$
	Thus, from \eqref{teo-2-003}, we obtain that
	$$
	\langle \mA \hat{u},v_0\rangle_{\mV}+\langle \mB^*\hat{p},v_0\rangle_{\mQ}=\langle \mA \hat{u},v_0\rangle_{\mV}=0,
	$$
	which implies that $A_{\mK\mK'}\hat{u}=0$, and hence $\hat{u}=0$.  Observing that \eqref{teo-2-003} becomes $\mB^*\hat{p}=0$, we use the third claim in Theorem \ref{teo-1} to conclude that $\hat{p}=0$. Hence, the operator \eqref{teo-2-matrix} is invertible.
	
	Let us rewrite \eqref{teo-2-001} as follows
	\begin{equation*}
		\boldsymbol{u}=\boldsymbol{F}+\int_{0}^t\boldsymbol{k}(t,s)\boldsymbol{u}(s)ds,
	\end{equation*}
	where
	$$\boldsymbol{u}=\begin{bmatrix}
		u\\p
	\end{bmatrix},\quad
	\boldsymbol{F}=\begin{bmatrix}
		\mA & \mB^*\\
		\mB & 0
	\end{bmatrix}^{-1}\begin{bmatrix}
		f\\g
	\end{bmatrix},\quad \boldsymbol{k}(t,s)=\begin{bmatrix}
		\mA & \mB^*\\
		\mB & 0
	\end{bmatrix}^{-1}\begin{bmatrix}
		k_1(t,s)\mA & k_2(t,s)\mB^*\\
		k_3(t,s)\mB & 0
	\end{bmatrix}.
	$$
	
	Then, observing that $\boldsymbol{F}\in L_{\mathcal{J}}^1(\mV\times\mQ)$ and $\boldsymbol{k}$ is continuous and bounded in $L^1(\mathcal{T})$, we invoke \cite{golden2013boundary} or  \cite[Chapter 2, Section 3]{gripenberg1990volterra} to conclude that there exists a unique $\boldsymbol{u}\in L_{\mathcal{J}}^1(\mV\times\mQ)$, solution to Problem \ref{prob3}. This concludes the proof.\qed
\end{proof}
The main relevance of  Theorem \ref{teo-2} is that provides the existence and uniqueness of solution to Problem \ref{prob4} due to its equivalence with Problem \ref{prob3}. 

The following result provides the data dependence  estimate of the solution of Problem \ref{prob4}.
\begin{teo}\label{teo5}
	Assume that Assumption \ref{assumption-1} holds. Let $\Ima\mB=\mQ'$. Moreover, assume that there exists a constant $\beta$ such that $\Vert \mathbb{B}^*q\Vert \geq \beta\Vert q\Vert_{\mQ}$ for all $q\in \mQ$, i.e., $b(\cdot,\cdot)$ satisfies an inf-sup condition. Then, for every $f\in L_{\mathcal{J}}^1(\mV')$ and $g\in L_{\mathcal{J}}^1(\mQ')$, Problem \ref{prob4} has a unique solution that satisfies
	\begin{equation*}
		\label{teo5-u-estimate}
		\Vert u\Vert_{L_{\mathcal{J}}^1(\mV)}+\Vert p \Vert_{L_{\mathcal{J}}^1(\mQ)} \leq \big(C_1+C_3\big)\Vert f\Vert_{L_{\mathcal{J}}^1(\mV')} + \big(C_2+C_4\big)\Vert g\Vert_{L_{\mathcal{J}}^1(\mQ')},
	\end{equation*}
	where 
	\begin{align*}
		\nonumber&C_1:=\frac{1}{\alpha_0}\bigg[1 + T\bigg(\frac{\Vert a\Vert}{\alpha_0}C_{\widetilde{k}} +C_{k_3}\bigg)e^{\displaystyle T\bigg(\frac{\Vert a\Vert}{\alpha_0}C_{\widetilde{k}} +C_{k_3}\bigg)}\bigg],\\
		\nonumber&C_2:= \frac{1}{\beta}\bigg[1+ \frac{\Vert a\Vert}{\alpha_0} \bigg]\bigg[1+T\bigg(\frac{\Vert a\Vert}{\alpha_0}C_{\widetilde{k}} +C_{k_3}\bigg)e^{\displaystyle T\bigg(\frac{\Vert a\Vert}{\alpha_0}C_{\widetilde{k}} +C_{k_3}\bigg)} \bigg],\\
		&C_3:=1+C_{k_2}e^{TC_{k_2}} + C_1\Vert a \Vert\bigg[1+C_{k_1} +C_{k_2}e^{\displaystyle TC_{k_2}}(1+TC_{k_1})\bigg],\\
		&C_4:=C_2\Vert a \Vert\bigg[1+C_{k_1} +C_{k_2}e^{\displaystyle TC_{k_2}}(1+TC_{k_1})\bigg],
	\end{align*}
	with  $C_{\widetilde{k}}$ satisfying $\vert k_1(t,s)-k_3(t,s)\vert\leq C_{\widetilde{k}}$.	
\end{teo}

\begin{proof}
	Let $u_g=\mathbb{L}_\mathbb{B}g$, where $\mathbb{L}_{\mathbb{B}}$ is the operator of Theorem \ref{teo-1}. From the continuity of $\mathbb{L}_{\mathbb{B}}$ we have that $\beta\Vert u_g\Vert_{\mV}\leq\Vert g\Vert_{\mQ'},$ and using this inequality together with the continuity of $\mathbb{A}$, we have that
	\begin{equation*}
		\Vert \mathbb{A} u_g\Vert_{\mV'}\leq \Vert a\Vert\,\Vert u_g\Vert_{\mQ'}\leq \frac{\Vert a \Vert}{\beta}\Vert g\Vert_{\mQ'}.
	\end{equation*}
	Set $$\displaystyle u_\mK=u-u_g - \int_{0}^tk_3(t,s)u(s)ds=u-\mathbb{L}_\mathbb{B}g-\int_{0}^tk_3(t,s)u(s)ds.$$ 
	Then 
	$$\mathbb{B}u_\mK=\mathbb{B}u - \mathbb{B}(\mathbb{L}_\mathbb{B}g) - \int_{0}^tk_3(t,s)\mB u(s)ds=0,$$
	which clearly implies that $u_\mK\in \mathcal{K}$. Moreover, note that $u_\mK$ solves the system
	\begin{equation*}
		\left\{
		\begin{aligned}
			\mA u_\mK + \mB^* p&=f - \mA u_g +\int_{0}^t\bigg[\widetilde{k}(t,s)\mA u(s) +k_3(t,s)\mB^* p(s)\bigg]\,ds\\
			\mB u_\mK&=0,
		\end{aligned}
		\right.
	\end{equation*}
	a.e. in $\mathcal{J}$, where $\widetilde{k}(t,s)=k_1(t,s)-k_3(t,s)$, or equivalently
	\begin{equation}
		\label{teo5-eq4}\left\{
		\begin{aligned}
			a(u_\mK,v) + b(v,p)&=\langle f,v\rangle_{\mV} - a(u_g,v) \\
			&\hspace{0.7cm}+ \int_{0}^{t}\bigg[\widetilde{k}(t,s)a(u(s),v) + k_3(t,s)b(v,p(s))\bigg]\,ds,\\
			b(u_\mK,q)&=0,
		\end{aligned}
		\right.
	\end{equation}
	for all $(v,q)\in\mV\times\mQ$.	Set  $v=u_\mK(t)$  in the first equation of \eqref{teo5-eq4}. This yields to $b(u_\mK,p)=\langle u_\mK, \mB^* p\rangle_{\mV}=\langle \mB u_\mK,p\rangle_{\mQ}=0$. Moreover, due the ellipticity of $a(\cdot,\cdot)$ in  $\mK$ we obtain
	\begin{equation*}
		\begin{aligned}
			\alpha_0\Vert u_\mK\Vert_{\mV}^2&\leq a(u_\mK,u_\mK)=\langle f,v\rangle_{\mV}  - a(u_g,u_\mK) + \int_{0}^{t}\widetilde{k}(t,s)a(u(s),u_\mK(t))ds\\
			&\leq \Vert f \Vert_{\mV'}\Vert u_\mK\Vert_{\mV} + \Vert a \Vert\,\Vert u_g\Vert_{\mV} \Vert u_{\mK}\Vert_\mV+ C_{\widetilde{k}}\Vert a \Vert \int_{0}^{t}\Vert u(s)\Vert_{\mV}ds\Vert u_\mK\Vert_{\mV}.
		\end{aligned}
	\end{equation*} 
	Thus,  we obtain 
	\begin{equation*}
		\alpha_0\Vert u_\mK\Vert_{\mV}
		\leq \Vert f\Vert_{\mV'} + \frac{\Vert a \Vert}{\beta}\Vert g \Vert_{\mQ'} + C_{\widetilde{k}}\Vert a \Vert\int_{0}^{t}\Vert u(s)\Vert_{\mV}\,ds.
	\end{equation*}
	Hence,
	\begin{equation*}
		\begin{aligned}
			\Vert u\Vert_{\mV}&=\bigg\Vert u_\mK+ u_g + \int_{0}^tk_3(t,s)u(s)ds\bigg\Vert_{\mV}\\
			&\leq \Vert u_\mK\Vert_{\mV} + \Vert u_g\Vert_{\mV} + C_{k_3}\int_{0}^t\Vert u(s)\Vert_{\mV}\,ds\\
			&\leq \frac{1}{\alpha_0}\bigg(\Vert f\Vert_{\mV'} + \frac{\Vert a \Vert}{\beta}\Vert g \Vert_{\mQ'} \bigg)  + \bigg(\frac{\Vert a\Vert}{\alpha_0}C_{\widetilde{k}} +C_{k_3}\bigg) \int_{0}^t\Vert u(s)\Vert_{\mV}\,ds + \frac{1}{\beta}\Vert g\Vert_{\mQ'}\\
			&= \frac{\Vert f\Vert_{\mV'}}{\alpha_0} + \frac{1}{\beta}\bigg(1+ \frac{\Vert a\Vert}{\alpha_0} \bigg)\Vert g\Vert_{\mQ'} +  \bigg(\frac{\Vert a\Vert}{\alpha_0}C_{\widetilde{k}} +C_{k_3}\bigg) \int_{0}^t\Vert u(s)\Vert_{\mV}\,ds.
		\end{aligned}
	\end{equation*}
	Applying Gronwall's Lemma to the above inequality, we obtain
	\begin{equation*}
		\begin{aligned}
			&\Vert u\Vert_{\mV}\leq \frac{\Vert f\Vert_{\mV'}}{\alpha_0} + \frac{1}{\beta}\bigg[1+ \frac{\Vert a\Vert}{\alpha_0} \bigg]\Vert g\Vert_{\mQ'}\\
			&+  \bigg(\frac{\Vert a\Vert}{\alpha_0}C_{\widetilde{k}} +C_{k_3}\bigg)e^{\displaystyle T \bigg(\frac{\Vert a\Vert}{\alpha_0}C_{\widetilde{k}} +C_{k_3}\bigg)}\int_{0}^{t}\bigg[\frac{\Vert f\Vert_{\mV'}}{\alpha_0} + \frac{1}{\beta}\bigg(1+ \frac{\Vert a\Vert}{\alpha_0} \bigg)\Vert g\Vert_{\mQ'}\bigg]\,ds\\
			&\leq \frac{\Vert f\Vert_{\mV'}}{\alpha_0} + \frac{1}{\beta}\bigg[1+ \frac{\Vert a\Vert}{\alpha_0} \bigg]\Vert g\Vert_{\mQ'}\\
			&+  \bigg(\frac{\Vert a\Vert}{\alpha_0}C_{\widetilde{k}} +C_{k_3}\bigg)e^{\displaystyle T \bigg(\frac{\Vert a\Vert}{\alpha_0}C_{\widetilde{k}} +C_{k_3}\bigg)}\bigg[\frac{\Vert f\Vert_{L_{\mathcal{J}}^1(\mV')}}{\alpha_0} + \frac{1}{\beta}\bigg(1+ \frac{\Vert a\Vert}{\alpha_0} \bigg)\Vert g\Vert_{L_{\mathcal{J}}^1(\mQ')}\bigg].\\
		\end{aligned}
	\end{equation*}
	After integrating in $\mathcal{J}$, we have an estimate for $u$
	\begin{equation*}
		\Vert u\Vert_{L_{\mathcal{J}}^1(\mV)}\leq C_1\Vert f\Vert_{L_{\mathcal{J}}^1(\mV')} + C_2\Vert g\Vert_{L_{\mathcal{J}}^1(\mQ')},
	\end{equation*}
	where
	\begin{align}
		\nonumber&C_1:=\frac{1}{\alpha_0}\bigg[1 + T\bigg(\frac{\Vert a\Vert}{\alpha_0}C_{\widetilde{k}} +C_{k_3}\bigg)e^{\displaystyle T\bigg(\frac{\Vert a\Vert}{\alpha_0}C_{\widetilde{k}} +C_{k_3}\bigg)}\bigg],\\
		\nonumber&C_2:= \frac{1}{\beta}\bigg[1+ \frac{\Vert a\Vert}{\alpha_0} \bigg]\bigg[1+T\bigg(\frac{\Vert a\Vert}{\alpha_0}C_{\widetilde{k}} +C_{k_3}\bigg)e^{\displaystyle T\bigg(\frac{\Vert a\Vert}{\alpha_0}C_{\widetilde{k}} +C_{k_3}\bigg)} \bigg].
	\end{align}
	
	On the other hand, the first equation in Problem \ref{prob3} implies that
	\begin{equation*}
		\begin{aligned}
			\Vert \mB^*p\Vert_{\mV'}&= \bigg\Vert f- \mA u +\int_{0}^{t}[k_1(t,s)\mA u(s) + k_2(t,s)\mB^*p(s)]ds \bigg\Vert_{\mV'}\\
			&\leq \Vert f\Vert_{\mV'} + \Vert a \Vert\, \Vert u\Vert_{\mV} + C_{k_1}\Vert a \Vert \int_{0}^t\Vert u(s)\Vert_{\mV}ds +C_{k_2}\int_0^t\Vert \mB^*p(s)\Vert_{\mV'}\,ds\\
			&\leq G_2(t)+C_{k_2}\int_0^t\Vert \mB^*p(s)\Vert_{\mV'}\,ds,
		\end{aligned}
	\end{equation*}
	where
	\begin{equation*}
		G_2(t):=\Vert f\Vert_{\mV'} + \Vert a \Vert\, \Vert u\Vert_{\mV} + C_{k_1}\Vert a \Vert\,\Vert u\Vert_{L_{\mathcal{J}}^1(\mV)}.
	\end{equation*}
	From Gronwall's lemma we obtain
	\begin{equation}
		\label{teo5-eq10}
		\begin{aligned}
			\Vert \mB^*p\Vert_{\mV'}&\leq G_2(t) +  C_{k_2}e^{TC_{k_2}}\int_{0}^{t}G_2(s)ds.\\
		\end{aligned}
	\end{equation}
	Integrating \eqref{teo5-eq10} in $\mathcal{J}$ we obtain that
	\begin{align*}
		\Vert \mB^*p\Vert_{L_{\mathcal{J}}^1(\mQ)}&\leq \bigg( 1+C_{k_2}e^{TC_{k_2}} \bigg)\Vert f\Vert_{L_{\mathcal{J}}^1(\mV')}\\
		&\hspace{1cm}+\Vert a \Vert\bigg[1+C_{k_1} +C_{k_2}e^{TC_{k_2}}(1+TC_{k_1})\bigg] \Vert u\Vert_{L_{\mathcal{J}}^1(\mV)}.
	\end{align*}
	From the inf-sup condition of $\mB$ we have that $\beta\Vert p\Vert_{\mQ}\leq \Vert \mB^*p\Vert_{\mV'}$, for all $p\in \mQ$. Thus, we  write an estimate for $p$ as
	\begin{equation*}
		\Vert p\Vert_{L_{\mathcal{J}}^1(\mQ)}\leq C_3\Vert f\Vert_{L_{\mathcal{J}}^1(\mV')} + C_4\Vert g\Vert_{L_{\mathcal{J}}^1(\mQ')},
	\end{equation*} 
	where $C_3$ and $C_4$ are defined as follows
	\begin{equation*}
		\begin{aligned}
			&C_3:=1+C_{k_2}e^{TC_{k_2}} + \Vert a \Vert C_1\bigg[1+C_{k_1} +C_{k_2}e^{TC_{k_2}}(1+TC_{k_1})\bigg],\\
			&C_4:=\Vert a \Vert C_2\bigg[1+C_{k_1} +C_{k_2}e^{TC_{k_2}}(1+TC_{k_1})\bigg].
		\end{aligned}
	\end{equation*}
	This concludes the proof.\qed
\end{proof}
\subsection{Semi-discrete abstract analysis}
This section deals with the semi-discrete version of the continuous problem considered in the previous section. The goal is to derive error estimates which depends on the conditions of  the spatial component in the variables involved. 

Inspired by \cite{boffi2013mixed}, we summarize the discrete framework of this reference in the following Assumption.

\begin{assunp}
	\label{assumption-discrete}
	Assume that Assumption \ref{assumption-1} holds. Let us assume the existence of two finite dimensional spaces $\mV_h$ and $\mQ_h$ such that $\mV_h\subset\mV$ and $\mQ_h\subset \mQ$, and let us consider the discrete kernels 
	
	\begin{equation*}
		\mK_h\equiv\ker\mB_h:=\bigg\{ v\in \mV_h\,\,:\,\, b(v,q)=0,\,\,\,\forall q\in\mQ_h\bigg\},
	\end{equation*}
	\begin{equation*}
		\mH_h\equiv\ker\mB_h^*:=\bigg\{q\in\mQ_h\,\,:\,\, b(v,q)=0,\,\,\,\forall v\in\mV_h\bigg\}.
	\end{equation*}
\end{assunp}

Following Assumption \ref{assumption-discrete},  we  consider the restriction to the discrete spaces of the bilinear forms $a(\cdot,\cdot)$ and $b(\cdot,\cdot)$. Moreover, in general the discrete kernels are not conforming subspaces of the continuous, i.e.,  $\mK_h\nsubseteq\mK$ and $\mH_h\nsubseteq\mH$.  Also, we will assume that  there exists a positive constant $\alpha_*^0>0$, independent of $h$, such that
\begin{equation}
	\label{discrete-ellipticity1}
	a(v_0,v_0)\geq\alpha_*^0\|v_0\|_{\mV}^2\,\,\,\,\forall v_0\in\mK_h. 
\end{equation}

On the other hand, for $b(\cdot,\cdot)$, we will assume that there exists a positive constant $\beta_*$, independent of $h$,  such that
\begin{equation}
	\label{discrete-inf-sup}
	\inf_{q\in \mQ_h}\sup_{v\in \mV_h}\frac{b(v,q)}{\Vert v\Vert_\mV\Vert q\Vert_\mQ}\geq\beta_*.
\end{equation}

Due the conformity of the finite element spaces, we have that $\mB(\mV_h)\subseteq\mQ_h $ and   $\mB_h v=\mB v$. In this case, we may say that $\mB_h$ is the restriction of $\mB$ to $\mV_h$. Similar considerations hold for  operators $\mB_h^*$ and $\mA_h$.

The spatial semidiscretization of the continuous problems introduced below are based on a Galerkin method, where the finite element spaces are generated from orthogonal or orthonormal basis functions. Then, from the ellipticity of $a(\cdot,\cdot)$ in the kernel of $b(\cdot,\cdot)$ and  the surjectivity of $\mB$,  it is sufficient to follow the ideas from \cite{saedpanah2016existence} in order to guarantee the existence and uniqueness of a finite element solution of the corresponding semi-discrete problems. Hence, the remaining task in this section are the semi-discrete error estimates.

As a starting point, we follow the strategy from \cite{brezzi1974existence}: we first consider general approximations $u_I$ and $p_I$ of $u$ and $p$, respectively, in $\mV_h$ and $\mQ_h$, in order to estimate the distance between  $(u_h,p_h)$ and $(u_I,p_I)$ in terms of the distance of $(u_I,p_I)$ from $(u,p)$. 



For the sake of simplicity, we define the errors  $\texttt{e}_u:=u_h-u$ and $\texttt{e}_p:=p_h-p$. Inspired by the analysis in \cite[Chapter 4]{miao1998finite}, we further write 
\begin{equation*}\label{errors}
	\eu=\xi_u-\eta_u, \,\,\,\,\,\,\,\ep=\xi_p-\eta_p,
\end{equation*}
where $\xi_u:=u_h-u_I,\,\xi_p:=p_h-p_I,\, \eta_u=u-u_I,$ and $\eta_p=p-p_I.$¸

The discretization by finite elements of Problem~\ref{prob4}, reads as follows:
\begin{problem}\label{prob4-discreto}
	Given $f\in L_{\mathcal{J}}^1(\mV')$ and $g\in L_{\mathcal{J}}^1(\mQ')$, find a pair $(u_h,p_h)\in L_{\mathcal{J}}^1(\mV_h\times \mQ_h)$ such that
	\begin{equation*}
		\label{weak-formulation-1-discrete}
		\left\{\begin{aligned}
			a(u_h,v) + b(v,p_h) &= \langle f,v\rangle_{\mV} + \int_{0}^{t}[k_1(t,s)a(u_h(s),v) + k_2(t,s)b(v,p_h(s))]ds,\\
			b(u_h,q)&=\langle g,q\rangle_{\mQ} + \int_{0}^tk_3(t,s)b(u_h(s),q)\,ds,
		\end{aligned}\right.
	\end{equation*}
	for all $(v,q)\in\mV_h\times\mQ_h$.
\end{problem}
From  the linearity of $a(\cdot,\cdot)$ and $b(\cdot,\cdot)$ and subtracting Problem \ref{prob4} and Problem \ref{prob4-discreto},  we obtain
\begin{equation}
	\label{semi-discrete-1}
	\left\{
	\begin{aligned}
		a(\eu,v) + b(v,\ep) &= \int_{0}^{t}[k_1(t,s)a(\eu(s),v) + k_2(t,s)b(v,\ep(s))]\,ds,\\
		b(\eu,q) &=\int_{0}^tk_3(t,s)b(\eu(s),q) ds,
	\end{aligned}
	\right.
\end{equation}
for all $(v,q)\in\mV_h\times\mQ_h$. Then, adding and subtracting $u_I$ and $p_I$, and according to Assumption \ref{assumption-discrete}, we conclude that for every $(u_I,p_I)\in L_{\mathcal{J}}^1(\mV_h\times\mQ_h)$, the pair $(\xi_u,\xi_p)\in L_{\mathcal{J}}^1(\mV_h\times\mQ_h)$ is the solution of
\begin{equation}
	\label{semi-discrete-2}
	\left\{
	\begin{aligned}
		a(\xi_u,v) + b(v,\xi_p) &=\langle \mF,v\rangle_{\mV}+ \int_{0}^{t}\!\!\bigg[k_1(t,s)a(\xi_{u}(s),v) \! + \! k_2(t,s)b(v,\xi_p(s)) \bigg]ds,\\
		b(\xi_u,q) &=\langle \mG,q\rangle_{\mQ}+\int_{0}^tk_3(t,s)b(\xi_u(s),q)ds,
	\end{aligned}
	\right.
\end{equation}
for all $(v,q)\in\mV_h\times\mQ_h$, where
\begin{equation}
	\label{eq:jesus1} 
	\langle\mF,v\rangle_{\mV}=a(\eta_u,v)+b(v,\eta_p)-\int_{0}^{t}\bigg[k_1(t,s)a(\eta_u(s),v)+k_2(t,s)b(v,\eta_p(s))\bigg]\,ds,
\end{equation}
and
\begin{equation}
	\label{eq:jesus2}
	\langle \mG,q\rangle_{\mQ} = b(\eta_u,q) - \int_{0}^tk_3(t,s)b(\eta_u(s),q)\,ds. 
\end{equation}	

Applying Theorem \ref{teo5} in \eqref{semi-discrete-2}  we verify \eqref{discrete-ellipticity1} and \eqref{discrete-inf-sup}. Indeed, assume that  Problem \ref{prob4} has a solution $(u,p)$ and let $(u_h,p_h)$ be the unique solution to the semi-discrete Problem \ref{prob4-discreto}. Then, for every $u_I\in L_{\mathcal{J}}^1(\mV_h)$ and for every $p_I\in L_{\mathcal{J}}^1(\mQ_h)$, the following estimates hold

\begin{equation}
	\label{teo-sd-1*}
	\begin{aligned}
		&\Vert \xi_{u}\Vert_{L_{\mathcal{J}}^1(\mV)}+\Vert \xi_p\Vert_{L_{\mathcal{J}}^1(\mQ)}\\
		&\hspace{2cm}\leq \max\bigg\{C_{1_*},C_{3_*}\bigg\}\Vert \mF\Vert_{L_{\mathcal{J}}^1(\mV')}+ \max\bigg\{C_{2_*},C_{4_*}\bigg\}\Vert \mG\Vert_{L_{\mathcal{J}}^1(\mQ')},
	\end{aligned}
\end{equation}
where $C_{i_*}, i=1,\ldots,4$ are positive constants defined by
\begin{align*}
	\nonumber&C_{1_*}:=\frac{1}{\alpha_*^0}\bigg[1 + T\bigg(\frac{\Vert a\Vert}{\alpha_*^0}C_{\widetilde{k}} +C_{k_3}\bigg)e^{T\bigg(\displaystyle\frac{\Vert a\Vert}{\alpha_*^0}C_{\widetilde{k}} +C_{k_3}\bigg)}\bigg],\\
	\nonumber&C_{2_*}:= \frac{1}{\beta_*}\bigg[1+ \frac{\Vert a\Vert}{\alpha_*^0} \bigg]\bigg[1+T\bigg(\frac{\Vert a\Vert}{\alpha_*^0}C_{\widetilde{k}} +C_{k_3}\bigg)e^{T\bigg(\displaystyle\frac{\Vert a\Vert}{\alpha_*^0}C_{\widetilde{k}} +C_{k_3}\bigg)} \bigg],\\
	&C_{3_*}:=1+C_{k_2}e^{TC_{k_2}} + \Vert a \Vert C_1\bigg[1+C_{k_1} +C_{k_2}e^{TC_{k_2}}(1+TC_{k_1})\bigg],\\
	&C_{4_*}:=\Vert a \Vert\bigg[1+C_{k_1} +C_{k_2}e^{TC_{k_2}}(1+TC_{k_1})\bigg]C_2.
\end{align*}

Finally,  we have that
\begin{align}
	\Vert \mF\Vert_{L_{\mathcal{J}}^1(\mV')}&\leq \bigg(1+T\max\bigg\{C_{k_1},C_{k_2}\bigg\}\bigg)\bigg(\Vert a\Vert\,\Vert \eta_u\Vert_{L_{\mathcal{J}}^1(\mV)} + \Vert b\Vert\,\Vert \eta_p\Vert_{L_{\mathcal{J}}^1(\mQ)} \bigg),\label{F_sd-estimate-1}\\
	\Vert \mG\Vert_{L_{\mathcal{J}}^1(\mQ')}&\leq \Vert b\Vert \Vert \eta_u\Vert_{L_{\mathcal{J}}^1(\mV)}\bigg(1+C_{k_3}\bigg),\label{G_sd-estimate-1}
\end{align}
which are a direct consequence of \eqref{eq:jesus1} and \eqref{eq:jesus2}, respectively.

With these results at hand, we are in position to prove  the following error estimate.

\begin{teo}[The basic error estimate]
	\label{teo-sd-basic-error1}
	Under Assumption \ref{assumption-discrete}, assume that $\mV_h$ and $\mQ_h$ verify the semi-discrete ellipticity in the kernel \eqref{discrete-ellipticity1} and the discrete inf-sup condition \eqref{discrete-inf-sup}, respectively. Assume that Problem \ref{prob4} has a unique solution $(u,p)$ and let $(u_h,p_h)$ be the unique solution to Problem \ref{prob4-discreto}. Then,  the following estimates hold
	\begin{equation*}
		\begin{aligned}
			\Vert \eu\Vert_{L_{\mathcal{J}}^1(\mV)}+\Vert\ep\Vert_{L_{\mathcal{J}}^1(\mV)}\leq \bigg(C_{1u}+C_{2u}\bigg) &\inf_{v\in \mV_h}\Vert u-v\Vert_{L_{\mathcal{J}}^1(\mV)}
			\\&+ \bigg(C_{1p}+C_{2p}\bigg)\inf_{q\in \mQ_h}\Vert p-q\Vert_{L_{\mathcal{J}}^1(\mQ)},\\
		\end{aligned}
	\end{equation*}
	Moreover, we have that there exists a constant $C>0$ depending on $C_{iu}$ and $C_{ip}, i=1,2$, and independent of $h$, such that
	\begin{equation*}
		\Vert u_h-u\Vert_{L_{\mathcal{J}}^1(\mV)} + \Vert p_h-p\Vert_{L_{\mathcal{J}}^1(\mQ)}\leq C\bigg(\inf_{v\in \mV_h}\Vert u-v\Vert_{L_{\mathcal{J}}^1(\mV)}+\inf_{q\in \mQ_h}\Vert p-q\Vert_{L_{\mathcal{J}}^1(\mQ)}\bigg).
	\end{equation*}
\end{teo}
\begin{proof}
	From \eqref{teo-sd-1*}, along with the estimates \eqref{F_sd-estimate-1} and \ref{G_sd-estimate-1} we have that
	\begin{equation*}
		\label{teo-sd-basic-error1-1}
		\begin{aligned}
			\Vert \eu&\Vert_{L_{\mathcal{J}}^1(\mV)}\leq \Vert \xi_{u} - \eta_u\Vert_{L_{\mathcal{J}}^1(\mV)}\leq \Vert \xi_{u}\Vert_{L_{\mathcal{J}}^1(\mV)}+\Vert \eta_u\Vert_{L_{\mathcal{J}}^1(\mV)}\\
			&\leq C_{1_*}\Vert \mF\Vert_{L_{\mathcal{J}}^1(\mV')}+ C_{2_*}\Vert \mG\Vert_{L_{\mathcal{J}}^1(\mQ')} + \Vert \eta_u\Vert_{L_{\mathcal{J}}^1(\mV)}\\
			&\leq C_{1_*}\bigg(1+T\max\bigg\{C_{k_1},C_{k_2}\bigg\}\bigg)\bigg[\Vert a\Vert\,\Vert \eta_u\Vert_{L_{\mathcal{J}}^1(\mV)}+ \Vert b\Vert\,\Vert \eta_p\Vert_{L_{\mathcal{J}}^1(\mQ)} \bigg] \\
			&\hspace{3cm} + C_{2_*}\Vert b\Vert(1+C_{k_3})\Vert \eta_u\Vert_{L_{\mathcal{J}}^1(\mV)}+ \Vert \eta_u\Vert_{L_{\mathcal{J}}^1(\mV)}\\
			&\leq \bigg[C_{1_*}\bigg(1+T\max\bigg\{C_{k_1},C_{k_2}\bigg\}\bigg)\Vert a\Vert + C_{2_*}\Vert b\Vert(1+C_{k_3})+1\bigg]\Vert \eta_u\Vert_{L_{\mathcal{J}}^1(\mV)}\\
			&\hspace{3cm} + C_{1_*}\bigg(1+T\max\bigg\{C_{k_1},C_{k_2}\bigg\}\bigg)\Vert b\Vert\,\Vert \eta_p\Vert_{L_{\mathcal{J}}^1(\mQ)}\\
			&= C_{1u} \Vert \eta_u\Vert_{L_{\mathcal{J}}^1(\mV)}  +C_{1p}\Vert \eta_p\Vert_{L_{\mathcal{J}}^1(\mQ)},
		\end{aligned}
	\end{equation*}
	where
	\begin{equation*}
		\begin{aligned}
			&C_{1u}:=C_{1_*}\bigg(1+T\max\bigg\{C_{k_1},C_{k_2}\bigg\}\Vert a\Vert + C_{2_*}\Vert b\Vert(1+C_{k_3})\bigg)+1,\\ &C_{1p}:=C_{1_*}\bigg(1+T\max\bigg\{C_{k_1},C_{k_2}\bigg\}\bigg)\Vert b\Vert.	
	\end{aligned}\end{equation*}
	Similarly, we have
	\begin{equation*}
		\label{teo-sd-basic-error1-2}
		\begin{aligned}
			\Vert \ep&\Vert_{L_{\mathcal{J}}^1(\mQ)}\leq \Vert \xi_p-\eta_p\Vert_{L_{\mathcal{J}}^1(\mQ)}\leq \Vert \xi_p\Vert_{L_{\mathcal{J}}^1(\mQ)}+ \Vert \eta_p\Vert_{L_{\mathcal{J}}^1(\mQ)}\\
			&\leq C_{3_*}\Vert \mF\Vert_{L_{\mathcal{J}}^1(\mV')} +C_{4_*}\Vert \mG\Vert_{L_{\mathcal{J}}^1(\mQ')}+\Vert \eta_p\Vert_{L_{\mathcal{J}}^1(\mQ)}\\
			&\leq C_{3_*}\bigg(1+T\max\bigg\{C_{k_1},C_{k_2}\bigg\}\bigg)\bigg[\Vert a\Vert\,\Vert \eta_u\Vert_{L_{\mathcal{J}}^1(\mV)}+ \Vert b\Vert\,\Vert \eta_p\Vert_{L_{\mathcal{J}}^1(\mQ)} \bigg] \\
			&\hspace{3cm} + C_{4_*}\bigg(\Vert b\Vert(1+C_{k_3})\Vert \eta_u\Vert_{L_{\mathcal{J}}^1(\mV)}\bigg)+ \Vert \eta_p\Vert_{L_{\mathcal{J}}^1(\mV)}\\
			&\leq \big[C_{3_*}\bigg(1+T\max\bigg\{C_{k_1},C_{k_2}\bigg\}\bigg)\Vert a\Vert + C_{4_*}\Vert b\Vert(1+C_{k_3})\big]\Vert \eta_u\Vert_{L_{\mathcal{J}}^1(\mV)}\\
			&\hspace{2cm} + \bigg[C_{3_*}\bigg(1+T\max\bigg\{C_{k_1},C_{k_2}\bigg\}\bigg)\Vert b\Vert+1\bigg]\Vert \eta_p\Vert_{L_{\mathcal{J}}^1(\mQ)}\\
			&= C_{1u} \Vert \eta_u\Vert_{L_{\mathcal{J}}^1(\mV)}  +C_{1p}\Vert \eta_p\Vert_{L_{\mathcal{J}}^1(\mQ)},
		\end{aligned}
	\end{equation*}
	where
	\begin{equation*}
		\begin{aligned}
			&C_{2u}:=C_{3_*}\bigg(1+T\max\bigg\{C_{k_1},C_{k_2}\bigg\}\Vert a\Vert + C_{4_*}\Vert b\Vert(1+C_{k_3})\bigg),\\ &C_{2p}:=C_{3_*}\bigg(1+T\max\bigg\{C_{k_1},C_{k_2}\bigg\}\bigg)\Vert b\Vert+1.	
		\end{aligned}
	\end{equation*}
	The proof is complete by adding the estimates and  taking the infimum on $\mV_h$ and $\mQ_h$.\qed
\end{proof}

\subsection{Error estimates in a weaker norm}


Let us consider two Hilbert spaces $\mV_{-}$ and $\mQ_{-}$ to be less regular spaces than $\mV$ and $\mQ$, respectively, with dense inclusions
\begin{equation}
	\label{dual-inclusion1}
	\mV \xhookrightarrow[]{}\mV_{-}\quad\text{and}\quad \mQ\xhookrightarrow[]{}\mQ_{-}.
\end{equation}

Our goal is to estimate $\Vert u-u_h\Vert_{L_{\mathcal{J}}^1(\mV_{-})}$ and $\Vert p - p_h\Vert_{L_{\mathcal{J}}^1(\mQ_{-})}$. To this end, let us denote,
\begin{equation*}
	\mV_{+}'=(\mV_{-})',\quad \mQ_{+}'=(\mQ_{-})',
\end{equation*}
where the subindex  $``+"$ suggests that we have more regular spaces. For instance, the inclusions provided in \eqref{dual-inclusion1} allow to obtain
\begin{equation*}
	\label{dual-inclusion2}
	\mV_{+}'\xhookrightarrow[]{}\mV',\quad \mQ_{+}'\xhookrightarrow[]{}\mQ'.
\end{equation*} 

For each $w\in \mV_{++}$, $w_h\in \mV_h$, $m\in\mQ_{++}$ and $m_h\in\mQ_h$, we define the errors  $\ew:=w-w_h$ and $\emm:=m-m_h$, where $\mV_{++}$ and $\mQ_{++}$ can be understood as more regular spaces than $\mV$ and $\mQ$, respectively, satisfying
\begin{equation*}
	\label{dual-inclusion3}
	\mV_{++}\xhookrightarrow[]{}\mV,\quad\mQ_{++}\xhookrightarrow[]{}\mQ.
\end{equation*}



For the forthcoming analysis, we  introduce the following hypothesis, concerning with a dual-backward mixed formulation of Problem \ref{prob4}.

\begin{hypo}
	\label{dual-hypo1-prob4}
	For any $\tau\in\mathcal{J}$ and for any $(f_+,g_+)\in L_{[0,\tau]}^{\infty}(\mV_{+}'\times\mQ_{+}')$, we assume that the solution $(w,m)$ of 
	\begin{equation}
		\label{dual-problem1}
		\left\{\begin{aligned}
			a(v,w(t))+b(v,m(t))&=\langle f_+,v\rangle_{\mV_{+}'\times \mV} \\
			&\hspace{0.5cm}+ \int_{t}^{\tau}\bigg[k_1(s,t)a(v,w(s))+k_2(t,s)b(v,m(s)) \bigg]ds,\\
			b(w(t),q)&=\langle g_+,q\rangle_{\mQ_{+}'\times \mQ}+\int_{t}^\tau k_3(t,s)b(w(s),q)\,ds,
		\end{aligned}\right.
	\end{equation}
	for all $(v,q)\in \mV\times\mQ$, belongs to $L_{[0,\tau]}^{\infty}(\mV_{++}\times\mQ_{++})$ a.e. in $[0,\tau]$. Moreover, assume that there exists a constant $\widehat{C}$, independent of $f_+$ and $g_+$, such that
	\begin{equation*}
		\label{dual-non-perturbed-bound1}
		\Vert w\Vert_{L_{[0,\tau]}^{\infty}(\mV_{++})}+\Vert m\Vert_{L_{[0,\tau]}^{\infty}(\mV_{++})}\leq \widehat{C}\bigg(\Vert f_+\Vert_{L_{[0,\tau]}^{\infty}(\mV_{+}')}+\Vert g_+\Vert_{L_{[0,\tau]}^{\infty}(\mQ_{+}')}\bigg).
	\end{equation*}
\end{hypo}
%
%
Hypothesis 1 gives, in practice, an additional regularity property  for the solution  $(w,m)\in L_{[0,\tau]}^{\infty}(\mV_{++}\times\mQ_{++})$ when the data is such that  $f\in L_{[0,\tau]}^{\infty}(\mV_{+}')$ and $g\in L_{[0,\tau]}^{\infty}(\mQ_{+}')$ (See \cite{shaw2001optimal} and the references therein). 
%

Now we will prove the main result of this section.
\begin{teo}
	\label{dual-teo1-prob4}
	Under the hypothesis of Theorem \ref{teo-sd-basic-error1}, assume that Hypothesis \ref{dual-hypo1-prob4} holds. Then, there exists a constant $C$, independent of $h$, such that
	\begin{equation*}
		\label{estimate1-dual-prob4}
		\begin{aligned}
			&\Vert u - u_h\Vert_{L_{\mathcal{J}}^1(\mV_{-})} + \Vert p - p_h\Vert_{L_{\mathcal{J}}^1(\mQ_{-})}\\
			&\hspace{2cm}\leq C\bigg(\inf_{v\in \mV_h}\Vert u - v\Vert_{L_{\mathcal{J}}^1(\mV)} + \inf_{q\in \mQ_h}\Vert p - q\Vert_{L_{\mathcal{J}}^1(\mQ)}\bigg)\bigg(r(h)+ n(h)\bigg)
		\end{aligned}
	\end{equation*}
	where
	$$
	\begin{aligned}
		&r(h):=\sup_{w\in L_{[0,\tau]}^{\infty}(\mV_{++})}\inf_{w_h\in L_{[0,\tau]}^{\infty}(\mV_h)}\frac{\Vert w- w_h\Vert_{L_{[0,\tau]}^{\infty}(\mV)}}{\Vert w \Vert_{L_{[0,\tau]}^{\infty}(\mV_{++})}},\\
		&n(h):=\sup_{m\in L_{[0,\tau]}^{\infty}(\mQ_{++})}\inf_{m_h\in L_{[0,\tau]}^{\infty}(\mQ_h)}\frac{\Vert m- m_h\Vert_{L_{[0,\tau]}^{\infty}(\mQ)}}{\Vert m \Vert_{L_{[0,\tau]}^{\infty}(\mQ_{++})}}.
	\end{aligned}
	$$
	Moreover, if $r(h)+n(h)\leq Ch$, then
	\begin{equation*}
		\begin{aligned}
			\Vert u - u_h\Vert_{L_\mathcal{J}^1(\mV_{-})} + &\Vert p - p_h\Vert_{L_\mathcal{J}^1(\mQ_{-})}\\
			&\leq Ch\left(\inf_{v\in \mV_h}\Vert u - v\Vert_{L_\mathcal{J}^1(\mV)} + \inf_{q\in \mQ_h}\Vert p - q\Vert_{L_\mathcal{J}^1(\mQ)}\right).
		\end{aligned}
	\end{equation*}
\end{teo}

\begin{proof}
	Taking time dependent test functions $v\in L^1(0,\tau;\mV)$ in the first equation of \eqref{dual-problem1}, and  integrating in $[0,\tau]$, we obtain
	\begin{equation}
		\label{dual-bound1}
		\begin{aligned}
			\int_{0}^\tau\langle f_+,v(t)\rangle_{\mV_{+}'\times \mV} dt&=\int_0^\tau\bigg\{a(v(t),w(t))+b(v(t),m(t)) \\
			&- \int_{0}^{t}\bigg[k_1(t,s)a(v(s),w(t))+k_2(t,s)b(v(s),m(t)) \bigg]\,ds\bigg\}dt,
		\end{aligned}
	\end{equation}
	by interchanging the order of integration in the Volterra integral. Also,  taking $q\in L^1(0,\tau;\mQ)$ in the second equation of \eqref{dual-problem1} we have
	\begin{equation}
		\label{dual-bound2}
		\begin{aligned}
			\int_{0}^\tau\langle g_+,q(t)\rangle_{\mQ_{+}'\times \mQ} dt&=\int_0^\tau\bigg\{b(w(t),q(t))-\int_{0}^t k_3(t,s) b(w(t),q(s))ds\bigg\}dt
		\end{aligned}
	\end{equation}
	Set  $v=u-u_h$ in \eqref{dual-bound1} and $q=p-p_h$ in \eqref{dual-bound2}. Then,
	\begin{align}
		&\int_{0}^\tau\langle f_+,\eu(t)\rangle_{\mV_{+}'\times \mV} dt=\int_0^\tau\bigg\{a(\eu(t),w(t))+b(\eu(t),m(t))\nonumber \\
		&\hspace{2.5cm}- \int_{0}^{t}\bigg[k_1(t,s)a(\eu(s),w(t))+k_2(t,s)b(\eu(s),m(t)) \bigg]\,ds\bigg\}dt,\label{dual-bound3}\\
		&\int_{0}^{\tau}\langle g_+,\ep(t)\rangle_{\mQ_{+}'\times \mQ}dt=\int_0^\tau\bigg\{b(w(t),\ep(t))-\int_{0}^t k_3(t,s) b(w(t),\ep(s))ds\bigg\}dt.\label{dual-bound4}
	\end{align}
	Adding and subtracting $w_h$ and $m_h$ in \eqref{dual-bound3} and \eqref{dual-bound4} yield to
	\begin{align}
		\nonumber&\int_{0}^\tau\langle f_+,\eu(t)\rangle_{\mV_{+}'\times \mV} dt=\int_0^\tau\bigg\{a(\eu(t),\ew(t))\!+\!b(\eu(t),\emm(t)) \!+\! a(\eu(t),w_h(t))\\
		&\hspace{1.5cm}\nonumber+b(\eu(t),m_h(t)) - \int_{0}^{t}\bigg[k_1(t,s)\bigg(a(\eu(s),\ew(t))+a(\eu(s),w_h(t))\bigg)\\
		&\hspace{1.5cm}+k_2(t,s)\bigg(b(\eu(s),\emm(t)) + b(\eu(s),m_h(t)) \bigg)\bigg]ds\bigg\}dt,\label{dual-bound5}
	\end{align}
	and
	\begin{align}
		\nonumber\int_{0}^{\tau}&\langle g_+,\ep(t)\rangle_{\mQ_{+}'\times \mQ}dt=\int_0^\tau\bigg\{b(\ew(t),\ep(t)) + b(w_h(t),\ep(t))\\
		&\hspace{1cm}-\int_{0}^t k_3(t,s)\bigg[b(\ew(t),\ep(s))+b(w_h(t),\ep(s))\bigg]\, ds\bigg\}dt.\label{dual-bound6}
	\end{align}
	Setting $f_+(t)=z_1(t)\eu(t)\Vert \eu(t)\Vert_{\mV_{-}}^{-1}$, with
	$$
	z_1(t)=1+\int_{0}^t\bigg[k_3(t,s)-k_2(t,s)\bigg]b(\eu(s),m_h)ds\Vert \eu(t)\Vert_{\mV_{-}}^{-1},
	$$
	gives that $\Vert f_+\Vert_{\mV_{-}}=\vert z_1(t)\vert$. Thus
	$$\langle f_+,\eu\rangle_{\mV_{+}'\times \mV_{-}}=\Vert \eu\Vert_{\mV_{-}} + \int_{0}^t\bigg[k_3(t,s)-k_2(t,s)\bigg]b(\eu(s),m_h)\,ds.$$
	Replacing the previous result in \eqref{dual-bound5}, yields to
	\begin{align}
		\nonumber\int_{0}^\tau&\Vert \eu(t)\Vert_{\mV_{-}}dt=\int_0^\tau\bigg\{a(\eu(t),\ew(t))+b(\eu(t),\emm(t)) + a(\eu(t),w_h(t)) \\
		&\nonumber\hspace{0.3cm}+b(\eu(t),m_h(t))- \int_{0}^{t}\bigg[k_1(t,s)a(\eu(s),\ew(t))+k_1(t,s)a(\eu(s),w_h(t))\\
		&\hspace{0.3cm}+k_2(t,s)b(\eu(s),\emm(t)) + k_3(t,s)b(\eu(s),m_h(t))\bigg]ds \bigg\}dt,\label{dual-teo1-001}
	\end{align}
	On the other hand, setting  $g_+(t)=z_2(t)\ep(t)\Vert \ep(t)\Vert_{\mQ_{-}}^{-1}$, with
	$$
	z_2(t):=1+\int_0^t\bigg[k_2(t,s)-k_3(t,s)\bigg]b\left(w_h(t),\ep(s)\right)ds\Vert \ep(t)\Vert_{\mQ_{-}}^{-1},
	$$
	and observing that $\Vert g_+(t)\Vert_{\mQ_{-}}=\vert z(t)\vert$, we have from \eqref{dual-bound6} that
	\begin{align}
		\nonumber\int_{0}^{\tau}\Vert \ep(t)\Vert_{\mQ_{-}}dt&=\int_0^\tau\bigg\{b(\ew(t),\ep(t)) + b(w_h(t),\ep(t))\\
		&-\int_{0}^t\bigg[ k_3(t,s)b(\ew(t),\ep(s))+k_2(t,s)b(w_h(t),\ep(s))\bigg]\,ds\bigg\}dt.\label{dual-teo1-002}
	\end{align}
	Also, from \eqref{semi-discrete-1} we know that
	\begin{equation}
		\label{semi-discrete-1*}
		\left\{
		\begin{aligned}
			a(\eu,w_h) + b(w_h,\ep) &= \int_{0}^{t}\bigg[k_1(t,s)a(\eu(s),w_h) + k_2(t,s)b(w_h,\ep(s))\bigg]ds,\\
			b(\eu,m_h) &=\int_{0}^tk_3(t,s)b(\eu(s),m_h) ds
		\end{aligned}
		\right.
	\end{equation}
	for all $(w_h,m_h) \in \mV_h\times\mQ_h$. Therefore, by adding \eqref{dual-teo1-001} and \eqref{dual-teo1-002} and combining this with \eqref{semi-discrete-1*}, we obtain
	\begin{equation*}
		\label{dual-teo-004}
		\begin{aligned}
			&\int_{0}^\tau\bigg(\Vert \eu\Vert_{\mV_{-}}+\Vert \ep\Vert_{\mQ_{-}}\bigg)\,dt=\int_0^\tau\bigg\{a(\eu,\ew)+b(\eu,\emm) +b(\ew,\ep)\\
			&- \int_{0}^{t}\bigg[k_1(t,s)a(\eu(s),\ew)+k_2(t,s)b(\eu(s),\emm)+k_3(t,s)b(\ew,\ep(s)) \bigg]\,ds\bigg\}dt\\
			&\leq\max\bigg\{\Vert a\Vert,\Vert b\Vert\bigg\}\int_{0}^\tau\bigg\{\Vert \eu(t)\Vert_{\mV}+\Vert \ep\Vert_{\mQ}\\
			&\hspace{0.1cm}+\max\bigg\{C_{k_1},C_{k_2},C_{k_3}\bigg\}\int_{0}^t\bigg[\Vert \eu(s)\Vert_{\mV}+\Vert\ep(s)\Vert_{\mQ}\bigg]ds \bigg\}\bigg(\Vert \ew\Vert_{\mV}+\Vert \emm\Vert_{\mQ}\bigg)dt\\
			&\leq I\bigg(\Vert \eu\Vert_{L_{[0,\tau]}^1(\mV)}+\Vert \ep\Vert_{L_{[0,\tau]}^1(\mQ)}\bigg)\bigg(\Vert\ew\Vert_{L_{[0,\tau]}^{\infty}(\mV)}+\Vert\emm\Vert_{L_{[0,\tau]}^{\infty}(\mQ)}\bigg),
		\end{aligned}
	\end{equation*}
	where $I:=\max\{\Vert a\Vert,\Vert b\Vert\}(1+TC_{k})$, $C_k:=\max\{C_{k_1},C_{k_2},C_{k_3} \}$. In the above argument, we have also used the continuity of the bilinear forms, the boundedness of the kernels, H\"older's inequality, and \cite[Lemma 3]{shaw2001optimal}. 
	
	On the other hand, from Hypothesis \ref{dual-hypo1-prob4} and the definition of $r(h)$ and $n(h)$, we have that there exists a positive constant $\widehat{C}$ such that
	$$\Vert\ew\Vert_{L_{[0,\tau]}^{\infty}(\mV)}+\Vert\emm\Vert_{L_{[0,\tau]}^{\infty}(\mQ)}\leq \widehat{C}\,\bigg[r(h)+n(h)\bigg]\bigg(\Vert f\Vert_{L_{[0,\tau]}^{\infty}(\mV_{+}')}+\Vert g\Vert_{L_{[0,\tau]}^{\infty}(\mQ_{+}')}\bigg).$$
	Then, we have that 
	\begin{equation}
		\label{dual01}
		\begin{aligned}
			&\Vert \eu\Vert_{L_{[0,\tau]}^1(\mV_{-})}+\Vert \ep\Vert_{L_{[0,\tau]}^1(\mQ_{-})}\\
			&\hspace{0.5cm}\leq \widetilde{C}\,\bigg(\Vert f\Vert_{L_{[0,\tau]}^{\infty}(\mV_{+}')}+\Vert g\Vert_{L_{[0,\tau]}^{\infty}(\mQ_{+}')}\bigg)\bigg(\Vert \eu\Vert_{L_{[0,\tau]}^1(\mV)}+\Vert \ep\Vert_{L_{[0,\tau]}^1(\mQ)}\bigg),
		\end{aligned}
	\end{equation}
	where $\widetilde{C}=I\widehat{C}\big[r(h)+n(h)\big]
	$. Then, taking $f_+$ and $g_+$ such that $\Vert f\Vert_{\mV_{+}'}=\Vert g_+\Vert_{\mQ_{+}'}=1$ gives a $L^1(0,\tau)$ estimate. Since $\tau$ is arbitrary, we can choose $\tau=T$ and apply Theorem \ref{teo-sd-basic-error1} in the left hand side of \eqref{dual01} to obtain the desired result.\qed
\end{proof}

\section{Applications to viscoelastic problems}
\label{sec:applications_timoshenko_beam}

In this section we will apply the theory developed above in two well known viscoelastic problems, considering an additional unknown, we obtain a mixed viscoelastic formulation that also fits in our study, as we will show in the following. Firstly, we will  apply our theory to the Laplace operator with memory terms, which has been previously trated in a similar way by \cite{sinha2009mixed}. Next, we will consider a mixed formulation for a Viscoelastic Timoshekno beam (\cite{payette2010nonlinear}). This is important from the  engineering point of view, since the thickness of the structure causes locking in numerical methods,  and we show that our viscoelastic formulation in mixed form avoids the phenomenon. 

We will start with some further notations.

Let $\Omega\subset\mathbb{R}^{n}$, with $n\in\{1,2\}$, be an open and convex domain with boundary $\partial\Omega$. We denote by $L^2(\Omega)$ and $H^1(\Omega)$  the usual Lebesgue and Sovolev spaces, endowed with the standard norm $\Vert \cdot \Vert_{L^2(\Omega)}$ and $\Vert \cdot \Vert_{H^1(\Omega)}$, respectively. We also consider $H_0^1(\Omega)$ the subspace of $H^1(\Omega)$ consisting of all functions which vanish at the boundary of $\Omega$ and $H(\text{div},\Omega)$ the space of vectorial functions with component in $L^2(\Omega)$ and with {\it divergence} also in $L^2(\Omega)$.  We introduce the spaces $\mathbb{Q}:=\{ (v\,,\eta\,)\in L^2(\Omega)\times L^2(\Omega) \}$ and $\mathbb{H}:=\{ (v\,,\eta\,)\in H^1(\Omega)\times H^1(\Omega) \}.$

Let $T\in[0,\infty)$ be such that we define the period of observation $\mathcal{J}:=[0,T]$. 

In the next examples, we will consider an adequate approximation of the continuous spaces by using finite element spaces defined on a family of triangulations of the domain $\Omega$ which mesh size $h$. We denote by DOF the number of degrees of freedom which correspond to the dimensions of the corresponding discrete spaces. To study the approximation error we will consider de corresponding norm defined by:
$$
\texttt{e}_0(f):=\Vert f-f_h\Vert_{L_{\mathcal{J}}^1(L^2(\Omega))}, \qquad \texttt{e}_1(f):=\Vert f- f_h\Vert_{L_{\mathcal{J}}^1(H^1(\Omega))}.
$$
for every function $f$ and its finite element approximations denoted by $f_h$. 
This allows to define an experimental rate of convergence  $\texttt{r}_i(\cdot)$ as
$$
\texttt{r}_i(\cdot):=\frac{\log(\texttt{e}_i(\cdot)/\texttt{e}_i'(\cdot))}{\log(h/h')},\;\; i=0,1,
$$
where $\texttt{e}_i$ and $\texttt{e}_i'$ denotes two consecutive relative errors and $h$ and $h'$ their corresponding mesh sizes.

\subsection{The Laplace problem with memory terms} 
\subsubsection{Mixed formulation}

Let $\Omega\subset\mathbb{R}^{n}$, with $n\in\{2,3\}$, be an open and convex domain with Lipschitz boundary $\partial\Omega$. The Laplace problem of our interest is the following: find $u\in L^1(\mathcal{J};H_0^1(\Omega))$ such that
$$
\left\{\begin{aligned}
	-\Delta u &=f+ \int_{0}^tk(t,s)\Delta u(s)&\text{ in } \Omega,\\
	u&=0&\text{ in }\partial\Omega.
\end{aligned}\right.
$$
Now, the space $
H(\dive,\Omega)=\{\sigma\in L^2(\Omega)^d\;:\; \dive\sigma\in L^2(\Omega) \}
$ is equipped with the norm $\Vert \sigma\Vert_{H(\dive,\Omega)}:=(\Vert \sigma\Vert_{L^2(\Omega)}^2+\Vert\dive\sigma\Vert_{L^2(\Omega)}^2)^{1/2}$.

Introducing the additional unknown $\sigma=\nabla u$, we obtain the following mixed formulation:
\begin{problem}
	\label{prob:laplace}
	Find $(\sigma,u)\in L^1(\mathcal{J};H(\dive;\Omega)\times L^2(\Omega))$ such that
	$$\left\{\begin{aligned}
		(\sigma,\tau) + (u,\dive\tau)&=0&\forall \tau\in H(\dive,\Omega),\\
		(\dive\sigma,v)&=-(f,v)+\int_{0}^tk(t,s)(\dive\sigma(s),v)ds&\forall v\in L^2(\Omega).
	\end{aligned}\right.$$
\end{problem}
Note that this mixed formulation is similar to that studied in \cite{sinha2009mixed}, but with the presence of the history of $\sigma$.This results in the abstract mixed formulation: Find $(\sigma,u)\in L^1(\mathcal{J},H(\dive,\Omega)\times L^2(\Omega))$ such that
$$
\left\{\begin{aligned}
	a(\sigma,\tau)+b(\tau,u)&=0&\forall \tau\in H(\dive,\Omega)\\
	b(\sigma,v)&=-(f,v)_{0,\Omega}+\int_{0}^tk(t,s)b(\sigma(s),v)ds&\forall v\in L^2(\Omega).
\end{aligned}\right.
$$
In this case, we have that $a(\cdot,\cdot)$ is the usual inner product in $L^2(\Omega)^2$ and $b(\tau,v)=(\tau,\dive v)$, which,  as is well known, satisfy the ellipticity in the kernel and inf-sup condition, respectively. Indeed, we note that, from the definition of $b(\cdot,\cdot)$, the continuous kernel is given by
$$
\mK:=\bigg\{\tau\in H(\dive,\Omega)\;:\;\dive\tau=0 \text{ in } \Omega \bigg\}.
$$
It follows that $a(\tau,\tau)=\Vert\tau\Vert_{L^2(\Omega)}^2=\Vert\tau\Vert_{H(\dive,\Omega)}^2$ for all $\tau\in\mK$. Hence we can apply directilly our theory (cf. Theorem \ref{teo5}) to conclude that this model has a unique solution $(\sigma,u)\in L^1(\mathcal{J};H(\dive;\Omega)\times L^2(\Omega))$ that satisfies
\begin{equation}
	\label{laplace1}
	\Vert\sigma\Vert_{L_\mathcal{J}^1(H(\dive,\Omega))}+\Vert u\Vert_{L_{\mathcal{J}}^1(L^2(\Omega))}\leq C\Vert f\Vert_{L^2(\Omega)},
\end{equation}
where $C>0$ is a constant that depends on the stability constants, the observation time and the history kernel. These types of formulations usually occur in fluid problems in porous media. $\sigma$ represents in velocity field and $u$ the concentration of solutes.

\subsubsection{Mixed Finite element approximations}
Given a quasi-uniform triangulation of $\Omega$, we denote by $V_h\times W_h$ the pair of finite element spaces of $H(\dive,\Omega)$ and $L^2(\Omega)$ such that
\begin{enumerate}
	\item $\dive V_h\subset W_h$,
	\item there exists a linear operator $\Pi_h:V\rightarrow V_h$ such that satisfies the commutative diagram property $\dive(\Pi_h(\cdot))=P_h(\nabla(\cdot))$, where $P_h:W\rightarrow W_h$ represents the $L^2(\Omega)$-projector.
\end{enumerate}
With this definition, the corresponding discrete formulation of Problem \ref{prob:laplace} reads as follows.
\begin{problem}
	\label{prob:laplace-discreto}
	Find $(\sigma_h,u_h)\in L^1(\mathcal{J};W_h\times V_h)$ such that
	$$\left\{\begin{aligned}
		(\sigma_h,\tau) + (u_h,\dive\tau)&=0&\forall \tau\in W_h,\\
		(\dive\sigma_h,v)&=-(f,v)+\int_{0}^tk(t,s)(\dive\sigma_h(s),v)ds&\forall v\in V_h.
	\end{aligned}\right.$$
\end{problem}
Moreover, we assume that the following approximation properties
\begin{equation}
	\label{laplace-estimates}
	\Vert \sigma - \Pi_h\sigma\Vert_{L^2(\Omega)}\leq Ch\Vert\sigma\Vert_{H^{1}(\Omega)}\quad\text{and}\quad
	\Vert u - P_h u\Vert_{L^2(\Omega)}\leq C h\Vert u \Vert_{H^1(\Omega)},
\end{equation}
are satisfied. This type of approximation property is found when Raviart-Thomas or BDM along with discontinuous piecewise constants elements are used. Further description of these estimates can be found in \cite{boffi2013mixed}.

Similar to the continuous case, we have the corresponding abstract semi-discrete formulation:
Find $(\sigma_h,u_h)\in L^1(\mathcal{J},W_h\times V_h)$ such that
$$
\left\{\begin{aligned}
	a(\sigma_h,\tau)+b(\tau,u_h)&=0&\forall \tau\in W_h\\
	b(\sigma_h,v)&=-(f,v)_{0,\Omega}+\int_{0}^tk(t,s)b(\sigma_h(s),v)ds&\forall v\in V_h.
\end{aligned}\right.
$$
The discrete kernel is given by
$$
\mK_h:=\bigg\{\tau_h\in W_h\;:\;\dive\tau_h=0 \text{ in } \Omega \bigg\},
$$
thus the bilinear form $a(\cdot,\cdot)$ is $\mK_h-$elliptic. For Raviart-Thomas elements, the discrete inf-sup condition from the Raviart-Thomas interpolation operator bound \eqref{laplace-estimates} and Fortin's lemma. Hence, the existence and uniqueness of a solution pair $(\sigma_h,u_h)\in L^1(\mathcal{J};W_h\times V_h)$ is guaranteed thanks to Theorem \ref{teo5}. Moreover, from Theorem \ref{teo-sd-basic-error1} we have that there exist $C>0$ such that
$$
\begin{aligned}
	\Vert \sigma-\sigma_h&\Vert_{L_\mathcal{J}^1(H(\dive,\Omega))}+\Vert u-u_h\Vert_{L_{\mathcal{J}}^1(L^2(\Omega))}\\
	&\leq C\big(\inf_{\tau\in W_h}\Vert \sigma-\tau\Vert_{L_\mathcal{J}^1(H(\dive,\Omega))}+\inf_{v\in V_h}\Vert u -v\Vert_{L_{\mathcal{J}}^1(L^2(\Omega))} \big)
\end{aligned}
$$
From the above, we can deduce the convergence rate of the semi-discrete finite element formulation, provided that the continuous solutions are smooth enough.
\begin{prop}
	\label{prop:laplace-convergence1}
	Let $(\sigma,u)\in L_{\mathcal{J}}^1(H(\dive,\Omega)\times L^2(\Omega))$ and $(\sigma_h,u_h)\in L_{\mathcal{J}}^1(W_h\times V_h)$ be the unique continuous and semi-discrete solutions to Problems \ref{prob:laplace} and \ref{prob:laplace-discreto}, respectively. Assume that $(\sigma,u)\in L^1(\mathcal{J};H^1(\Omega)\times H^1(\Omega))$ and $f\in L_{\mathcal{J}}^1(L^2(\Omega))$, then
	\begin{equation}
		\label{laplace-aproximation}
		\begin{aligned}
			&\|(\sigma,u)-(\sigma_h,u_h)\Vert_{L_{\mathcal{J}}^1(H(\dive,\Omega)\times L^2(\Omega))}\leq Ch \Vert f\Vert_{L_{\mathcal{J}}^1(L^2(\Omega))},
		\end{aligned}
	\end{equation}
	where  $C$ is a positive constant independent of $h$. 
\end{prop}
\begin{proof}
	From \eqref{laplace-estimates} and \eqref{laplace-aproximation} we obtain
	\begin{equation*}
		\label{laplace-convergence}
		\begin{aligned}
			\Vert (\sigma,u)-&(\sigma_h,u_h)\Vert_{L_{\mathcal{J}}^1(H(\dive,\Omega)\times L^2(\Omega))} \\
			&\leq C\Vert (\sigma,u)-(\Pi_h\sigma,\mathcal{P}_h u)\Vert_{L_{\mathcal{J}}^1(H(\dive,\Omega)\times L^2(\Omega))}\leq Ch\Vert f\Vert_{L_{\mathcal{J}}^1(L^2(\Omega))},
		\end{aligned}
	\end{equation*}
	\qed
\end{proof}
\subsubsection{Numerical example}
We will consider an  example from the non-fickian flow study in \cite{barbeiro2013laplace}. Here, the computational domain is $\Omega=(0,1)\times(0,1)$, the kernel is $k(t-s)=\frac{1}{\delta}e^{-(t-s)/\delta}$ and we consider  $\delta=0.01$. The function $f$ and boundary conditions are such that, the exact solution is
$$
u(x,y,t)=\cos(t)(x-x^2)(y-y^2).
$$
The observation time is $T=4.5$s and $3000$ time steps are used. In Table \ref{table7:error_sigma_u} we report the computed error and experimental rates of convergence. Using sufficiently small steps to reduce the incidence of the trapezoidal rule in the aproximation of the time integral guarantee a convergence order similar to that of mixed-laplacian models. For instance, on Figure \ref{fig:laplace_u} we observe the comparison of the maximum values reached of each function in the domain along the observation time, coinciding with the exact solution. This can be compared with Figure \ref{fig:laplace_sigma_u}, where we present  snapshots of the field $\sigma$ and $u$ in different time steps.
\begin{table}[!ht]
	\caption{Error data and rate of convergence of the velocity field $\sigma$ and $u$ in the steady non-fickian flow model.}
	\centering\begin{tabular}{CCCCCC}
		\hline
		\text{DOF}&h&\texttt{e}_0(\sigma)&\texttt{r}_0(\sigma)&\texttt{e}_0(u)&\texttt{r}_0(u)\\
		\hline
		259 &0.2020  &0.0311 &--    &0.0074  &--   \\
		423 &0.1571  &0.0244 &0.97  &0.0057  &0.98   \\
		744 &0.1178  &0.0184 &0.98  &0.0043  &0.99 \\
		1656&0.0785 &0.0123 &0.98  &0.0029  &0.99  \\
		2928&0.0589 &0.0092 &0.98  &0.0021  &0.99\\
		4560&0.0471 &0.0074 &0.98  &0.0017  &0.98  \\
		\hline
	\end{tabular}
	\label{table7:error_sigma_u}
\end{table}
\begin{figure}[!ht]
	\centering
	\includegraphics[scale=0.18]{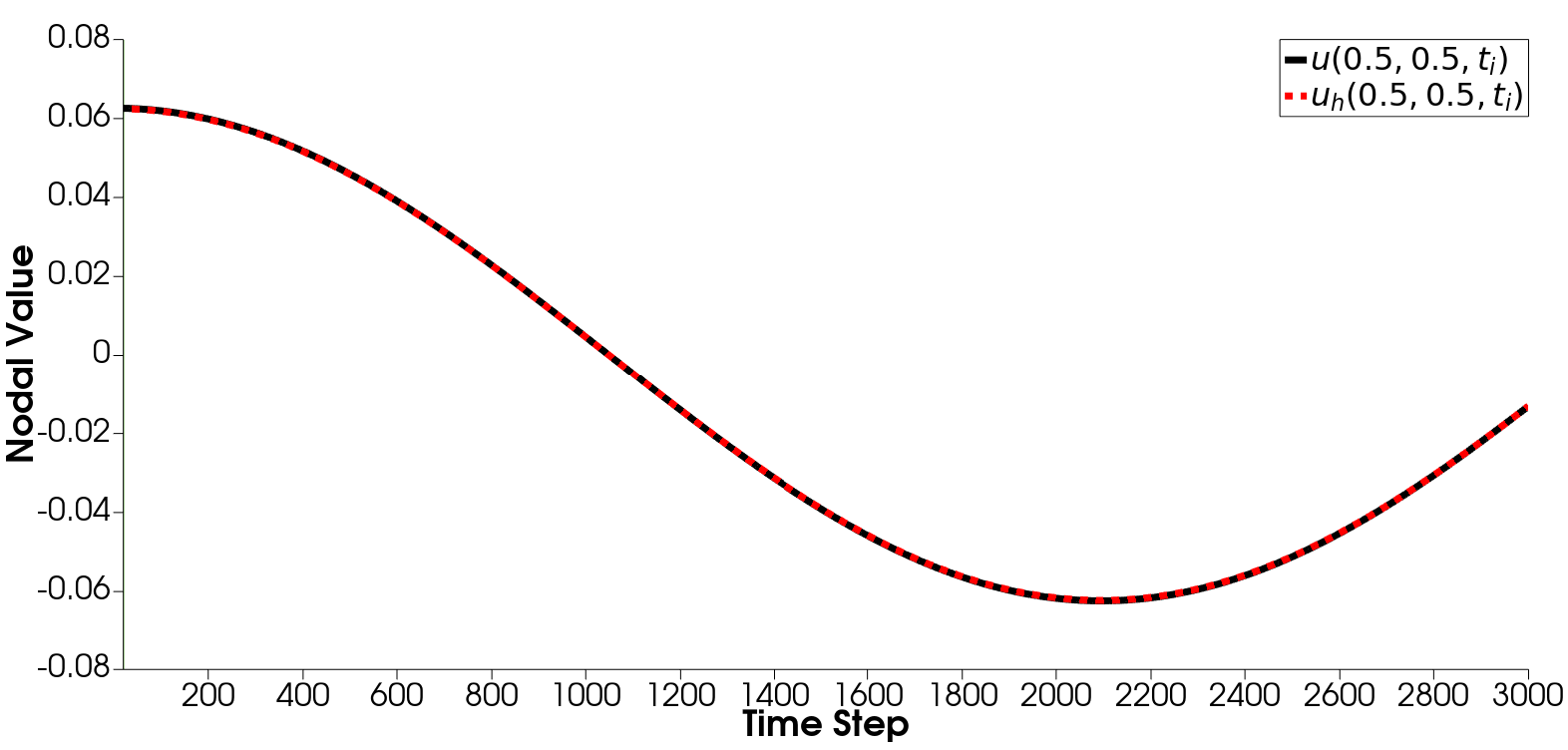}
	\caption{Comparison of the nodal values of $u$ and $u_h$ in $(0.5,0.5)$ in the non-fickian flow example. The observation time is $T=4.5$\,s and $\Delta t=0.0015$.}
	\label{fig:laplace_u}
\end{figure}
\begin{figure}[!ht]
	\begin{minipage}[c]{0.44\linewidth}
		\begin{center}
			\includegraphics[scale=0.12, trim={0 0 0 2.5cm},clip]{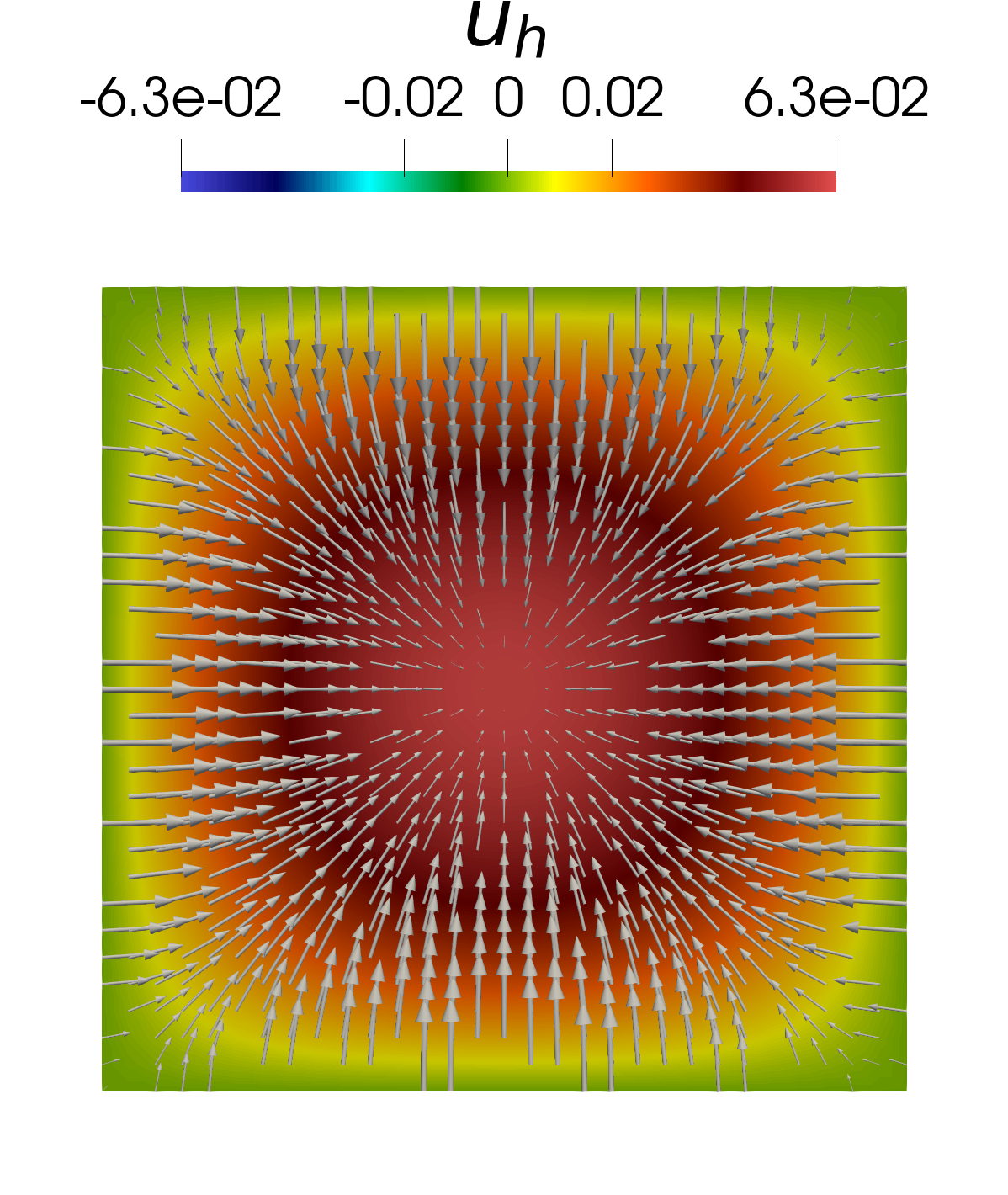}\\
			{$T=0$\,s}
		\end{center}
	\end{minipage}
	\begin{minipage}[c]{0.44\linewidth}
		\begin{center}
			\includegraphics[scale=0.12, trim={0 0 0 2.5cm},clip]{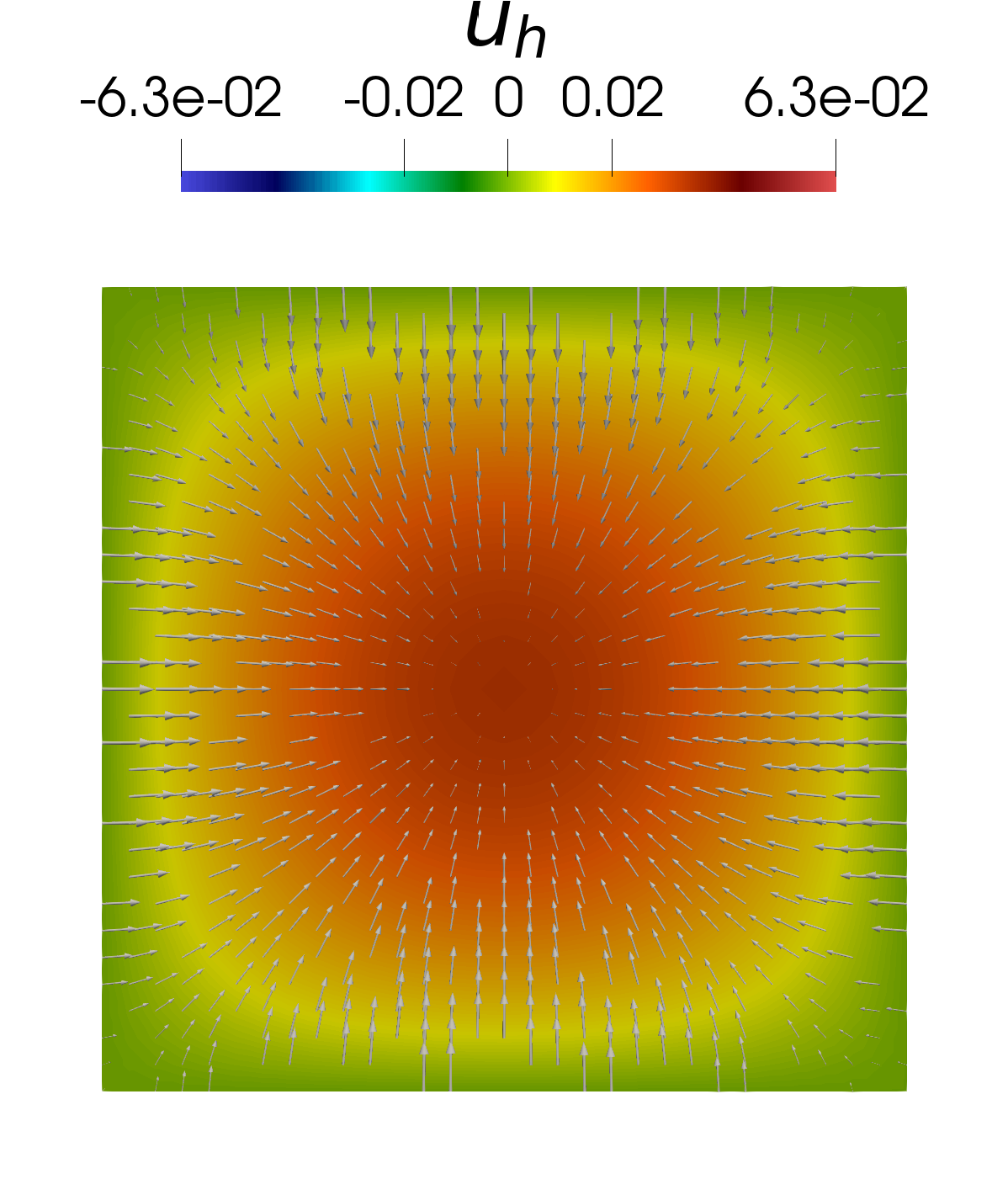}\\
			{$T=1$\,s}
		\end{center}
	\end{minipage}\\
	\begin{minipage}[c]{0.44\linewidth}
		\begin{center}
			\includegraphics[scale=0.12, trim={0 0 0 2.5cm},clip]{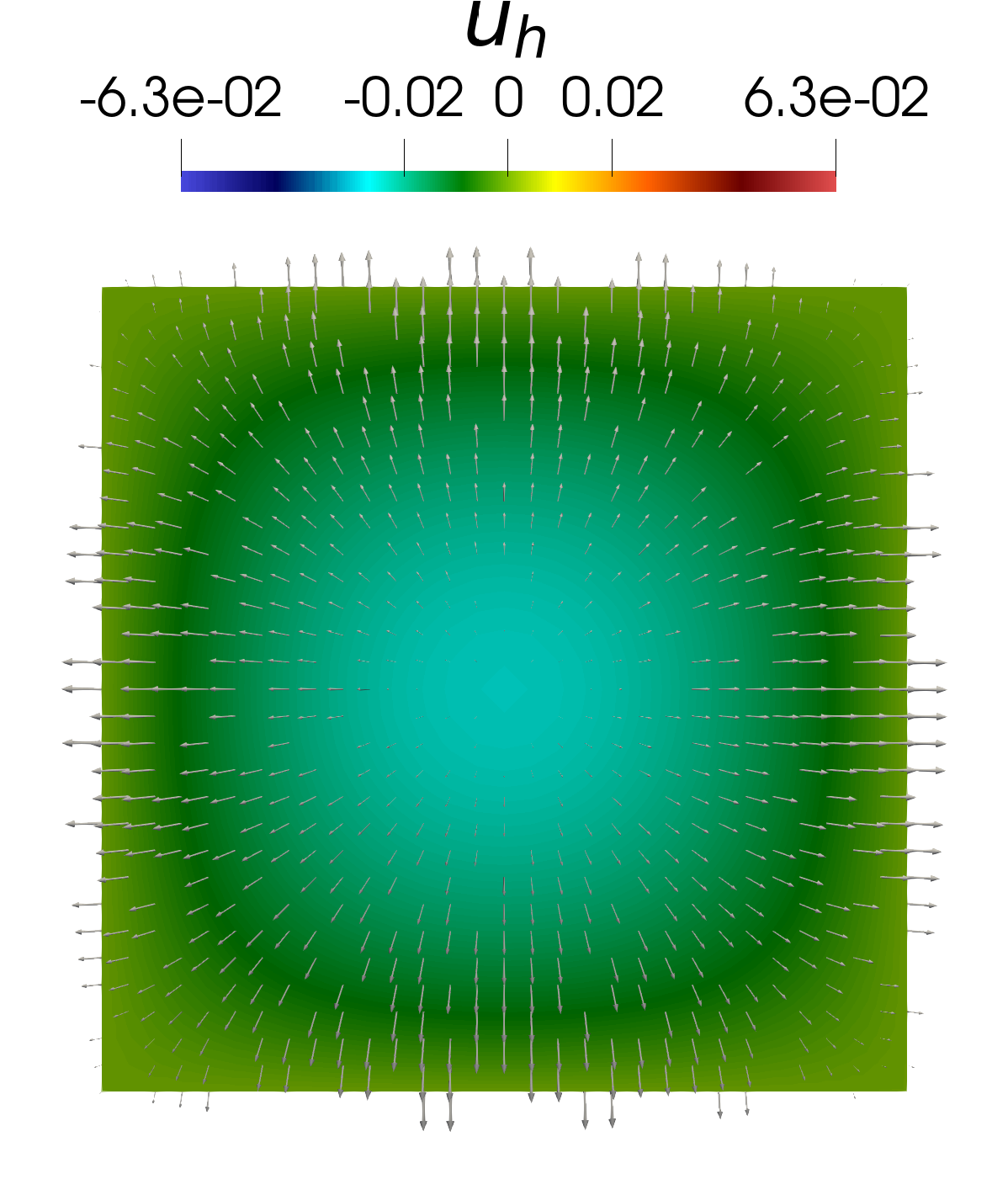}\\
			{$T=2$\,s}
		\end{center}
	\end{minipage}
	\begin{minipage}[c]{0.44\linewidth}
		\begin{center}
			\includegraphics[scale=0.12, trim={0 0 0 2.5cm},clip]{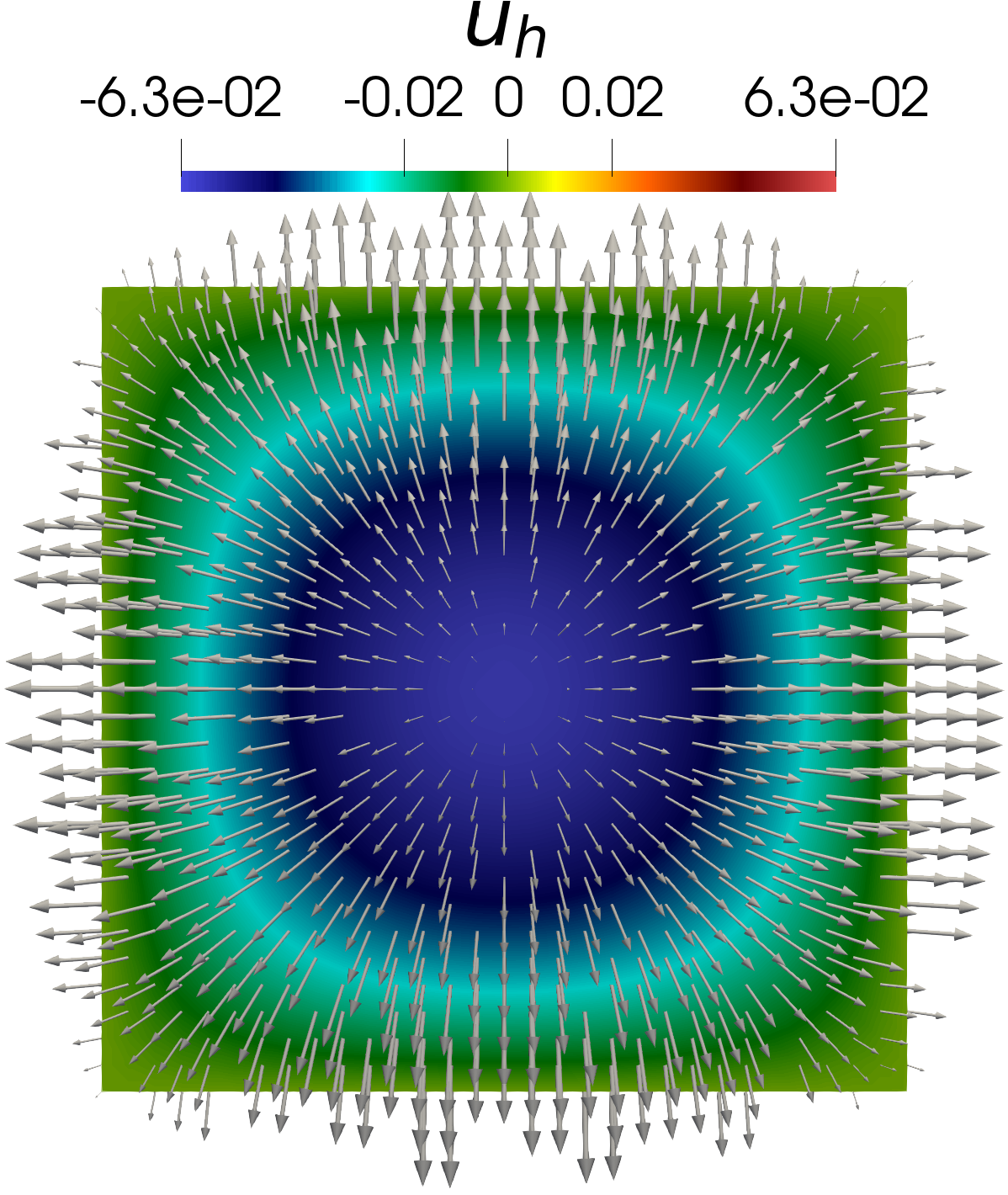}\\
			{$T=3$\,s}
		\end{center}
	\end{minipage}
	\caption{Evolution of the discrete vector field $\sigma_h$ (arrows) and $u_h$ (surface color). The selected time step is $\Delta t=0.0015$. }
	\label{fig:laplace_sigma_u}
\end{figure}

\subsection{The  viscoelastic Timoshenko beam problem}

\subsubsection{Mixed viscoelastic formulation}

We consider a viscoelastic Timoshenko beam clamped at the ends, which model  is obtained by considering the linear case in \cite{payette2010nonlinear}, i.e, find $(w,\beta)\in L_{\mathcal{J}}^1(H_0^1(\Omega)\times H_0^1(\Omega))$ such that
\begin{equation}
	\label{s-m-b-timoshenko1}
	\begin{aligned}
		&E(0)(I(x) \beta', \eta' )_{0,\Omega} + k_sG(0)\big( A(x)(\beta-w'),\eta-v'\big)_{0,\Omega} =(\widetilde{f},v)_{0,\Omega}  \\
		&+(\widetilde{g},\eta)_{0,\Omega}+\int_0^t\dot{E}(t-s) (I(x)\beta'(s),\eta\,' )_{0,\Omega}\,  ds\\
		&+k_s\int_0^t \dot{G}(t-s)\big(A(x) (\beta(s) -w'(s)),\eta\,-v' \big)_{0,\Omega} \, ds,
	\end{aligned}
\end{equation}
for all $(v\,,\eta\,)\in H_0^1(\Omega)\times H_0^1(\Omega)$, where $\Omega:=[0,L]$, $L$ is the length of the beam, $w$ is the displacement of the beam, $\beta$ represent the rotations,  $k_s$ is the correction factor, $I(x)$ is the moment of inertia of the cross-section, and $A(x)$ is the area of the cross-section. The load $\widetilde{f}$ represents a distributed transverse load, while $\widetilde{g}$  represents a moment load, and $(\widetilde{f},v)_{0,\Omega}:=\int_\Omega \widetilde{f}v\, dx $. The $L^1(\cdot)$ time regularity is assumed in this study because the main goal of this work is to obtain a locking free  numerical method,  whose essence lies in the space regularity. 

Let us introduce the following classic non-dimensional parameter, characteristic of the thickness of the beam
\begin{equation*}
	\varepsilon^2=\frac{1}{L}\int_\Omega \frac{I(x)}{A(x)L^2}dx,
\end{equation*}
which is assumed to be independent of time and is such that  $\varepsilon\in(0,\varepsilon_{\max}]$. Therefore, scaling the loads as $\widetilde{f}(x,t)=\varepsilon^3 f(x,t),$ $\widetilde{g}(x,t)=\varepsilon^3 g(x,t) $, with $f(x,t)$ and $g(x,t)$ independent of $\varepsilon$, dividing \eqref{s-m-b-timoshenko1} by $E(0)$ and defining
$\hat{I}(x):=\frac{I(x)}{\varepsilon^3}$ and $ \hat{A}(x):=k_s\frac{A(x)}{\varepsilon},$
we have that \eqref{s-m-b-timoshenko1} is equivalently to the following: 
\begin{problem}
	\label{prob-timoshenko-2}
	Given $f$, $g\in L_{\mathcal{J}}^1(L^2(\Omega))$, find $(\beta,w)\in L_{\mathcal{J}}^1(H_0^1(\Omega)\times H_0^1(\Omega))$ such that	
	\begin{equation*}
		\label{s-m-b-timoshenko4-qs2}
		\begin{aligned}
			&\big(\hat{I} \beta', \eta' \big)_{0,\Omega} + \frac{\varepsilon^{-2}}{2(1+\nu)}\big( \hat{A}(\beta-w'),\eta-v'\big)_{0,\Omega} =(f_E,v)_{0,\Omega} +(g_E,v)_{0,\Omega} \\
			&\hspace{0.2cm}+\int_0^tk(t,s)\bigg[  (\hat{I}\beta'(s),\eta\,' )_{0,\Omega} +\frac{\varepsilon^{-2}}{2(1+\nu)}\big(\hat{A} (\beta(s) -w'(s)),\eta\,-v' \big)_{0,\Omega} \bigg] ds,
		\end{aligned}
	\end{equation*}
	for all $(v\,,\eta\,)\in H_0^1(\Omega)\times H_0^1(\Omega)$, where $k(t,s)=\dot{E}(t-s)/E(0)$,  $f_E=f/E(0)$ and $g_E=g/E(0)$.
\end{problem}

To apply our  abstract framework, we consider the bending moment formulation studied in \cite{lepe2014locking} for a steady Timoshenko beam, where the following time-dependent variables are introduced:
$$
M:= \hat{I}\beta',\qquad V:=\kappa \varepsilon^{-2}(\beta- w'),
$$
which corresponds to the bending moment and shear stress, respectively. Here $\kappa:=\hat{A}/(2(1+\nu))$. Then, by standard argument used in  \cite{lepe2014locking}, we can rewritte  the problem  in the mixed variational form:
\begin{problem}
	\label{LMR1-MixedFormulation}
	Find $\big((M,V),(\beta,w)\big)\in L^1\big(\mathcal{J};\mathbb{H}\times\mathbb{Q}\big)$ such that
	\begin{equation*}
		\left\{\begin{aligned}
			&( M' - V,\eta)_{0,\Omega}  - ( V',v)_{0,\Omega} = -( f_E, v)_{0,\Omega} -(g_E,\eta)_{0,\Omega}  \\
			&\hspace{2cm}+ \int_{0}^{t}k(t,s)\bigg[( M'(s) - V(s),\eta)_{0,\Omega}  - ( V'(s),v)_{0,\Omega} \bigg]\,ds,\\
			&\big( M/\hat{I}, \tau)_{0,\Omega}  +\varepsilon^2\kappa^{-1}( V,\xi\big)_{0,\Omega}   +\big( \beta,\tau'-\xi\big)_{0,\Omega}  -\big( w, \xi'\big)_{0,\Omega} =0,
		\end{aligned}\right.
	\end{equation*}
	for all $\big((\tau,\xi);(\eta,v)\big)\in\mathbb{H}\times\mathbb{Q}$.
\end{problem} 

Note that Problem \ref{LMR1-MixedFormulation} can be writte in the following abstract form: 
Find $((M,V),(\beta,w))\in L^1\big(\mathcal{J};\mathbb{H}\times \mathbb{Q}\big)$ such that
\begin{equation}
	\label{LMR1-MixedFormulation-bilinear_forms}
	\left\{\begin{aligned}
		&a((M,V),(\tau,\xi)) + b((\tau,\xi),(\beta,w)) = 0\\
		&b((M,V),(\eta,v))=F(\eta,v) + \int_{0}^{t}k(t,s)b\big((M(s),V(s)),(\eta,v)\big)ds,
	\end{aligned}\right.
\end{equation}
where
\begin{align*}
	&a\big((M,V),(\tau,\xi)\big):=\big( M,\tau\big)_{0,\Omega}  +\varepsilon^2\kappa^{-1}( V,\xi)_{0,\Omega} ,\\
	&b\big((\tau,\xi),(\eta,v)\big):=( n, \tau'-\xi)_{0,\Omega}  - ( v,\xi')_{0,\Omega} ,\label{LMR1-bilinear-forms2}
\end{align*}
and $F(\eta,v):=-( f_E,v)_{0,\Omega} -(g_E,\eta)_{0,\Omega} $.

Then, according to \cite[Lemma 2.1 and Lemma 2.2]{lepe2014locking} we have that the bilinear form $a(\cdot,\cdot)$ is elliptic in $\mK:=\{ (\tau,\xi)\in\mathbb{H}\,:\, \tau'-\xi=0,\text{ and }\xi'=0 \text{ in }\Omega\}=\{(\tau,\tau'):\tau\in\mathbb{P}_1(\Omega)\}$ and $b(\cdot,\cdot)$ satisfies an inf-sup condition, with constants independent of $d$. Hence, we invoke Theorem \ref{teo5}, with $k_1(\cdot,\cdot)=k_2(\cdot,\cdot)=0$, $ k_3(t,s)=k(t,s)=\dot{E}(t-s)/E(0)$, $\mV=\mathbb{H}$, and $\mQ=\mathbb{Q}$. Then, we apply it to the system \eqref{LMR1-MixedFormulation-bilinear_forms} to conclude that there exists a constant $C>0$ depending on the stability constants and observation time, but not on the thickness parameter $\varepsilon$, such that
\begin{equation*}
	\label{lmr1-stability}
	\Vert (M,V)\Vert_{L_{\mathcal{J}}^1(\mathbb{H})} + \Vert (\beta,w)\Vert_{L_{\mathcal{J}}^1(\mathbb{Q})}\leq C\bigg(\Vert f_E\Vert_{L_{\mathcal{J}}^1(L^2(\Omega))}+\Vert g_E\Vert_{L_{\mathcal{J}}^1(L^2(\Omega))}\bigg).
\end{equation*}

To obtain the convergence results, we assume that the following additional regularity estimate for Problem \ref{LMR1-MixedFormulation} (see \cite[Proposition 2.1]{lepe2014locking} in the elastic case.)

\begin{equation*}
	\label{LMR1-additional-estimate}
	\begin{aligned}
		\Vert M\Vert_{L_{\mathcal{J}}^1(H^2(\Omega))}+\Vert V\Vert_{L_{\mathcal{J}}^1(H^2(\Omega))} + \Vert &\beta\Vert_{L_{\mathcal{J}}^1(H^1(\Omega))}+\Vert w\Vert_{L_{\mathcal{J}}^1(H^1(\Omega))}\\
		&\leq C\bigg(\Vert f_E\Vert_{L_{\mathcal{J}}^1(H^1(\Omega))}+\Vert g_E\Vert_{L_{\mathcal{J}}^1(H^1(\Omega))}\bigg).
	\end{aligned}
\end{equation*}

\subsubsection{Mixed finite element analysis}

Let $\mathscr{T}_h=\big\{\Omega_i \big\}_{i=1}^{n}$  a partition of $\Omega$ such that $\Omega_i=]x_{i-1},x_i[$, with length $h_i=x_i-x_{i-1}$, $\bigcap_{i=1}^n\Omega_i=\emptyset$ and  $\Omega=\bigcup_{i=1}^n\Omega_i$, $i=1,\dots,n$. Let $h=\max_{1\leq i \leq n}h_i$. We consider the following finite element spaces
\begin{align*}
	\mV_h&:=\bigg\{v\in H^1(\Omega)\;:\; v_{\vert \Omega_i}\in \mathbb{P}_1(\Omega_i),\, \Omega_i\in\mathscr{T}_h \bigg\},\\
	\mQ_h&:=\bigg\{q\in L^2(\Omega)\;:\; v_{\vert \Omega_i}\in \mathbb{P}_{0}(\Omega_i),\, \Omega_i\in\mathscr{T}_h \bigg\},
\end{align*}
and  $\mathbb{H}_h:=\mV_h\times \mV_h$. Then, the semi-discrete counterpart of Problem \ref{prob-timoshenko-2} reads as follows:

\begin{problem}
	\label{prob-timoshenko-1-discreto}
	Given $(f,g)\in L_{\mathcal{J}}^1(L^2(\Omega)\times L^2(\Omega))$, find $(\beta_h,w_h)\in L_{\mathcal{J}}^1(\mathbb{H}_h)$ such that	
	\begin{equation*}
		\label{s-m-b-timoshenko4-qs-discreto}
		\begin{aligned}
			&\big(\hat{I} \beta_h', \eta' \big)_{0,\Omega} + \frac{\varepsilon^{-2}}{2(1+\nu)}\big( \hat{A}(\beta_h-w_h'),\eta-v'\big)_{0,\Omega} =(f_E,v)_{0,\Omega} +(g_E,v)_{0,\Omega} \\
			&\hspace{0.2cm}+\int_0^tk(t,s)\bigg[  (\hat{I}\beta_h'(s),\eta\,' )_{0,\Omega} +\frac{\varepsilon^{-2}}{2(1+\nu)}\big(\hat{A} (\beta_h(s) -w_h'(s)),\eta\,-v' \big)_{0,\Omega} \bigg] ds,
		\end{aligned}
	\end{equation*}
	for all $(v\,,\eta\,)\in\mathbb{H}_h$, where again $k(t,s)=\dot{E}(t-s)/E(0)$, $f_E=f/E(0)$ and $g_E=g/E(0)$.
\end{problem}

Let  $\mathbb{Q}_h:=\mQ_h\times \mQ_h$, then, the semi-discrete counterpart of \eqref{LMR1-MixedFormulation-discreto} reads:
\begin{problem}
	\label{prob7}
	Find $\big((M_h,V_h),(\beta_h,w_h)\big)\in L_{\mathcal{J}}^1(\mathbb{H}_h\times\mathbb{Q}_h)$ such that
	\begin{equation}
		\label{LMR1-MixedFormulation-discreto}
		\left\{\begin{aligned}
			&\big( M_h' - V_h,\eta\big)_{0,\Omega}  - \big( V_h',v\big)_{0,\Omega} = -( f_E, v)_{0,\Omega}  - (g_E,\nu)_{0,\Omega}  \\
			&\hspace{2cm}+ \int_{0}^{t}k(t,s)\bigg[\big( M_h'(s) - V_h(s),\eta\big)_{0,\Omega}  - ( V_h'(s),v)_{0,\Omega} \bigg]ds,\\
			&\big( M_h, \tau\big) _{0,\Omega} +\varepsilon^2\kappa^{-1}\big( V_h,\xi\big)_{0,\Omega}   +\big( \beta_h,\tau'-\xi\big)_{0,\Omega}  -\big( w_h, \xi'\big)_{0,\Omega} =0,
		\end{aligned}\right.
	\end{equation}
	for all $((\tau,\xi),(\eta,v))\in \mathbb{H}_h\times\mathbb{Q}_h$.
\end{problem} 

This formulation is equivalent to the following abstract scheme:\linebreak Find $\big((M_h,V_h),(\beta_h,w_h)\big)\in L_{\mathcal{J}}^1(\mathbb{H}_h\times\mathbb{Q}_h)$ such that
\begin{equation*}
	\label{LMR1-MixedFormulation-bilinear_forms-discreto}
	\left\{\begin{aligned}
		&a\big((M_h,V_h),(\tau,\xi)\big) + b\big((\tau,\xi),(\beta_h,w_h)\big) = 0,\\
		&b\big((M_h,V_h),(\eta,v)\big)=F(\eta,v) + \int_{0}^{t}k(t,s)b\big((M_h(s),V_h(s)),(\eta,v)\big)ds,
	\end{aligned}\right.
\end{equation*}
for all $((\tau,\xi),(\eta,v))\in \mathbb{H}_h\times\mathbb{Q}_h$.

According to \cite{lepe2014locking}, the discrete null space of $b(\cdot,\cdot)$ , namely, 
$$
\mK_h:=\bigg\{(\tau,\xi)\in\mathbb{H}_h\,:\, b\big((\tau,\xi),(\eta_h,v)\big)=0, \forall (\eta_h,v)\in\mathbb{Q}_h \bigg\}, 
$$
coincides with $\mK$. Hence,   we have that $a(\cdot,\cdot)$ is $\mK_h$-elliptic and the discrete inf-sup condition on $b(\cdot,\cdot)$ is  satisfied. Moreover, the corresponding constants of the ellipticity and inf-sup conditions are independent of $h$ and $\varepsilon$. Thus, from Theorem \ref{teo-sd-basic-error1} we have that there exists $C>0$, independent of the thickness paramater $d$, such that
\begin{equation}
	\label{LMR1-discrete-estimate1}
	\begin{aligned}
		\Vert &(M,V)-(M_h,V_h)\Vert_{L_{\mathcal{J}}^1(\mathbb{H})} + \Vert (\beta,w)-(\beta_h,w_h)\Vert_{L_{\mathcal{J}}^1(\mathbb{Q})}\\
		&\leq C\bigg(\inf_{(\tau,\xi)\in \mathbb{H}_h}\Vert (M,V)-(\tau,\xi)\Vert_{L_{\mathcal{J}}^1(\mathbb{H})}+\inf_{(\eta,v)\in \mathbb{Q}_h}\Vert (\beta,w)-(\eta,v)\Vert_{L_{\mathcal{J}}^1(\mathbb{Q})}\bigg).
	\end{aligned}
\end{equation}
Then, by using standard argument on the right hand side of \eqref{LMR1-discrete-estimate1}, we obtain the convergence rate of the mixed finite element formulation.
\begin{prop}
	\label{prop:LMR1-convergence1}
	Let $((M,V),(\beta,w))\in L_{\mathcal{J}}^1(\mathbb{H}\times\mathbb{Q})$ and $\big((M_h,V_h),(\beta_h,w_h)\big)\in L_{\mathcal{J}}^1(\mathbb{H}_h\times\mathbb{Q}_h)$ be the continuous and semi-discrete solutions to Problems \ref{LMR1-MixedFormulation} and \eqref{LMR1-MixedFormulation-discreto}, respectively. If $(f,g)\in L_{\mathcal{J}}^1(H^1(\Omega)\times H^1(\Omega))$, then
	\begin{equation*}
		\begin{aligned}
			&\|(M,V)-(M_h,V_h)\Vert_{L_{\mathcal{J}}^1(\mathbb{H})} + \Vert (\beta,w)-(\beta_h,w_h)\Vert_{L_{\mathcal{J}}^1(\mathbb{Q})}\\
			&\hspace{5cm}\leq Ch \bigg(\Vert f_E\Vert_{L_{\mathcal{J}}^1(H^1(\Omega))}+\Vert g_E\Vert_{L_{\mathcal{J}}^1(H^1(\Omega))}\bigg),
		\end{aligned}
	\end{equation*}
	where  $C$ is a positive constant independent of $h$ and $\varepsilon$. 
\end{prop}

Now, by using the weak norm estimates studied in the previous section,  we provide additional convergence rate for the bending moment and the shear stress variables of the viscoelastic Timoshenko beam.

\begin{prop}
	\label{prop:LMR1-convergence2}
	Under the hypothesis of Proposition \ref{prop:LMR1-convergence1}, assume that the solution $\big(\widetilde{M},\widetilde{V},\widetilde{\beta},\widetilde{w}\big)$ of the dual-backward problem of Problem \ref{prob7} belongs to $L_{[0,\tau]}^{\infty}(\mathbb{H}\times\mathbb{H})$. Then, there exists $C>0$, independent of $h$ and $\varepsilon$,  such that
	\begin{equation*}
		\Vert (M,V)-(M_h,V_h)\Vert_{L_{\mathcal{J}}^1(\mathbb{Q})} \leq Ch^2\bigg(\Vert f_E\Vert_{L_{\mathcal{J}}^1(H^1(\Omega))}+\Vert g_E\Vert_{L_{\mathcal{J}}^1(H^1(\Omega))}\bigg).
	\end{equation*}
\end{prop}
\begin{proof}
	We start the proof by noting that the dual-backward form of Problem \ref{prob7} takes the form: \textit{find $(\widetilde{M},\widetilde{V},\widetilde{\beta},\widetilde{w})\in L_{[0,\tau]}^{\infty}(\mathbb{H}\times\mathbb{Q})$ such that a.e. in $[0,\tau]$,}
	\begin{equation}
		\label{LMR1-MixedFormulation-bilinear_forms-dual}
		\left\{\begin{aligned}
			&a((\tau,\xi),(\widetilde{M},\widetilde{V})) + b((\tau,\xi),(\widetilde{\beta},\widetilde{w})) = 0\\
			&b((\widetilde{M},\widetilde{V}),(\eta,v))=F_+(\eta,v) + \int_{\tau}^{t}k(s,t)b\big((\widetilde{M}(s),\widetilde{M}(s)),(\eta,v)\big)ds.
		\end{aligned}\right.
	\end{equation}
	Setting $\rho:=\tau - t$ and $\chi=\tau - s$ and defining $\overline{w}(\cdot):=\widetilde{w}(\tau - \cdot),\;\overline{\beta}(\cdot):=\widetilde{\beta}(\tau - \cdot),\;\overline{M}(\cdot):=\widetilde{M}(\tau - \cdot),\;\overline{V}(\cdot):=\widetilde{V}(\tau - \cdot),\; \overline{f}_{+}(\rho):=f_{+}(\tau -\rho)$ and $\overline{k}(\rho,\chi):=k(\tau-\chi,\tau - \rho),$  we have that the system \eqref{LMR1-MixedFormulation-bilinear_forms-dual} can be rewritten as: \textit{find $\left(\overline{M},\overline{V},\overline{\beta},\overline{w}\right)\in L_{[0,\tau]}^{\infty}(\mathbb{H}\times\mathbb{Q})$ such that for a.e. $\rho\in[0,\tau]$, }
	\begin{equation}
		\label{LMR1-MixedFormulation-bilinear_forms-dual2}
		\left\{\begin{aligned}
			&a\big((\tau,\xi),(\overline{M},\overline{V})\big) + b\big((\tau,\xi),(\overline{\beta},\overline{w})\big) = 0\\
			&b\big((\overline{M},\overline{V}),(\eta,v)\big)=\overline{F}_+(\eta,v) + \int_{0}^{\rho}\overline{k}(\rho,\chi)b\big((\overline{M}(\chi),\overline{M}(\chi)),(\eta,v)\big)d\chi.
		\end{aligned}\right.
	\end{equation}
	Noting that we have the same bilinear forms, we have from \cite[Lemma 4]{shaw2001optimal} that there exist a unique solution $\big(\overline{M},\overline{V},\overline{\beta},\overline{w}\big)\in L_{[0,\tau]}^{\infty}(\mathbb{H}\times\mathbb{Q})$ for \eqref{LMR1-MixedFormulation-bilinear_forms-dual2} (resp. a unique solution  $\big(\widetilde{M},\widetilde{V},\widetilde{\beta},\widetilde{w}\big)\in L_{[0,\tau]}^{\infty}(\mathbb{H}\times\mathbb{Q})$ for \eqref{LMR1-MixedFormulation-bilinear_forms-dual} ). Hence it follows that \eqref{LMR1-MixedFormulation-bilinear_forms-dual} satisfies Hypothesis \ref{dual-hypo1-prob4}, and the result follows directly from Proposition \ref{prop:LMR1-convergence1} and Theorem \ref{dual-teo1-prob4}.
\end{proof}
\subsubsection{Numerical tests}
\label{sec:numerical_examples}
Now  we will report numerical experiments in order to asses the performance of the finite element method for the  viscoelastic  Timoshenko beam by considerg both, quasi-static and dynamical cases, and different Timoshenko beam configurations are considered. 

In this study, the experimental nature of the relaxation modulus are replaced by assumed values of spring constants and viscosity parameters in order to consider the \textit{Standard Linear Solid model (SLS)}, where the relaxation modulus is given by the truncated Prony series:
\begin{equation*}
	\label{relaxation-modulus-SLS}
	E(t)=\frac{k_1k_2}{k_1+k_2}+\bigg(k_1-\frac{k_1k_2}{k_1+k_2} \bigg)e^{-t/\tau}.
\end{equation*}
Here,  $\tau=\eta_2/(k_1+k_2)$. Also, 
we consider a Poisson ration $\nu=0.35$ for all the experiments where it is involved.

\vspace{,5cm}
$\bullet$ {\it Rigidily joined beams}. 
For this experiment, we consider a fully clamped composed beam formed by two rigidly joined beams, which was studied in \cite{lovadina2011locking} and it is shown in Figure \ref{fig:viganouniforme1}.
\begin{figure}[!ht]
	\centering
	\includegraphics[scale=0.8]{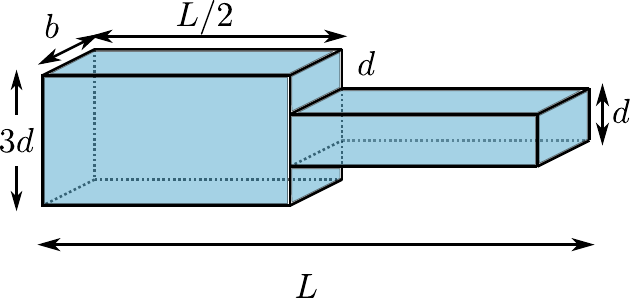}
	\caption{Rigidly joined beams.}
	\label{fig:viganouniforme1}
\end{figure}
The beam length is $L=1\,m$, with $b=0.03\,m$. The area of the cross-section and moment of inertia are given as follows.
$$
A(x)\!=\!\left\{\begin{aligned}
	&9\times 10^{-2}d\;\mbox{m}^2,&\!\!\!\!0\leq x\leq 0.5\\
	&3\times 10^{-2}d\;\mbox{m}^2, &\!\!\!\!0.5< x\leq 1.
\end{aligned}\right.
\hspace{0.3cm}
I(x)\!=\!\left\{\begin{aligned}
	&\frac{27\times10^{-2}d^3}{4}\mbox{m}^4,&\!\!\!0\leq x\leq 0.5\\
	&\frac{10^{-2}d^3}{4}\mbox{m}^4,&\!\!\!0.5< x\leq 1.
\end{aligned}\right.
$$
Thus, the thickness parameter is $\varepsilon^2=5d^2/12L^2$. Following the ideas  from \cite{lovadina2011locking}, we have taken meshes with an even number of elements such that the point $x=L/2$ is always a node. To compute the exact solution, we write an elastic problem following the work of \cite{lepe2014locking} and \cite{celiker2006locking}, then we solve the equilibrium equations by assuming clamped boundary conditions, and finally apply the corresponding principle. The beam is subject to a uniform distributed load  $q(x,t)=e^x\,H(t)\,\text{N/m}$. The selected parameters for the SLS are taken to be $k_1=k_2=1$\,N/m$^2$, $\eta_2=1$\,N$\cdot$s/m$^2$. The observation time is $15\,s$ and we choose a step size $\Delta t=0.003$, corresponding to $5000$ time steps. The selected thickness is set to be $d=0.001\,m$. Tables \ref{table5:error_l2_M_V_d001} and \ref{table5:error_l2_beta_w_d001} shows the corresponding rates when the system is solved using the bending moment formulation. The expected convergence rates in both cases are observed Figure \ref{fig:joined-beams-errors}
\begin{figure}[!ht]
	\centering
	\includegraphics[scale=0.5]{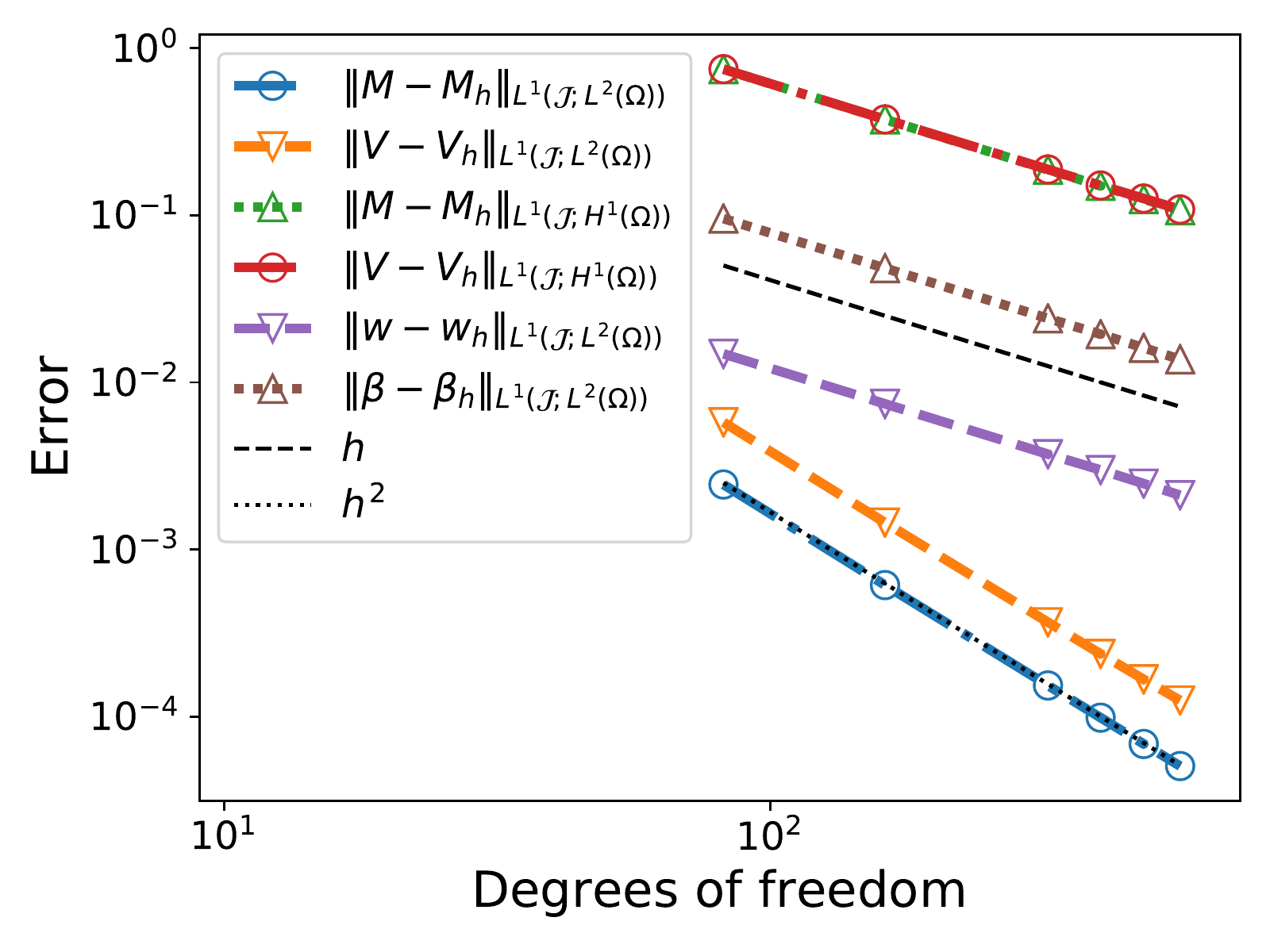}
	\caption{Convergence behavior for the rigidly joined clamped beams when solved with the bending moment formulation. The thickness is set to be $d=0.001\,m$.}
	\label{fig:joined-beams-errors}
\end{figure}
\begin{table}[!ht]
	\caption{Error data and rate of convergence of the bending moment $M$ and shear $V$ for a selected thickness $d=0.001\,m$ in the rigidly joined beams.}
	\centering\begin{tabular}{CCCCCC}
		\hline
		\text{DOF}&h&\texttt{e}_0(M)&\texttt{r}_0(M)&\texttt{e}_0(V)&\texttt{r}_0(V)\\
		\hline
		42 &0.05    &2.4456e-3 &--    &5.7650e-3  &--      \\
		82 &0.025   &6.1175e-4 &1.99  &1.4470e-3  &1.99   \\
		162&0.0125  &1.5337e-4 &1.99  &3.6648e-4  &1.98  \\
		202&0.0100  &9.8416e-5 &1.98  &2.3705e-4  &1.95 \\
		242&0.00833 &6.8517e-5 &1.98  &1.6692e-4  &1.92 \\
		282&0.00714 &5.0522e-5 &1.97  &1.2486e-4  &1.88  \\
		\hline
		\text{DOF}&h&\texttt{e}_1(M)&\texttt{r}_1(M)&\texttt{e}_1(V)&\texttt{r}_1(V)\\
		\hline
		42 &0.05    &7.4824e-1 &--    &7.4845e-1 &--   \\
		82 &0.025   &3.7413e-1 &0.99  &3.7454e-1 &0.99   \\
		162&0.0125  &1.8711e-1 &0.99  &1.8790e-1 &0.99 \\
		202&0.0100  &1.4966e-1 &1.00  &1.5062e-1 &0.99\\
		242&0.00833 &1.2479e-1 &0.99  &1.2591e-1 &0.98\\
		282&0.00714 &1.0700e-1 &0.99  &1.0825e-1 &0.97 \\
		\hline
	\end{tabular}
	
	\label{table5:error_l2_M_V_d001}
\end{table}
\begin{table}[!ht]
	\caption{Error data and rate of convergence of the transverse displacement $w$ and rotation $\beta$ for a selected thickness $d=0.001\,m$ in the rigidly joined beams.}
	\centering\begin{tabular}{CCCCCC}
		\hline
		\text{DOF}&h&\texttt{e}_0(w)&\texttt{r}_0(w)&\texttt{e}_0(\beta)&\texttt{r}_0(\beta)\\
		\hline
		42 &0.05    &1.4806e-2 &--    &9.5487e-2  &--   \\
		82 &0.025   &7.4304e-3 &0.99  &4.8188e-2  &0.98    \\
		162&0.0125  &3.7139e-3 &1.00  &2.4118e-2  &0.99 \\
		202&0.0100  &2.9870e-3 &0.97  &1.9401e-2  &0.97  \\
		242&0.00833 &2.4761e-3 &1.02  &1.6084e-2  &1.02\\
		282&0.00714 &2.1113e-3 &1.03  &1.3715e-2  &1.03  \\
		\hline
	\end{tabular}
	\label{table5:error_l2_beta_w_d001}
\end{table}

\vspace{,5cm}
$\bullet$ {\it Beam with smoothly changing moment of inertia and cross-section area}
Let us consider a fully clamped beam of length $L=1\,m$, with moment of inertia and cross-section area given by $I(x)=e^x/12$ and $A(x)=12e^{-x}$, respectively. The thickness parameter is given by $\varepsilon\approx0.14894$. The relaxation modulus is given by $E(t)=0.5(1+e^{-t})$, corresponding to the SLS parameters $k_1=k_2=1$\,N/m$^2$ and $\eta_2=1$\,N$\cdot$s/m$^2$. The quasi-static exact solution is again obtained by following \cite{celiker2006locking} and the corresponding principle. The beam is subjected to a varying load $q(t)=e^x\,\text{H(t)\,N/m}$. The chosen period of observation is $T=15\,s$ with $5000$ time steps. 
\begin{table}[!ht]
	\caption{Error data and rate of convergence of the bending moment $M$ and hear forces $V$ in the beam with smoothly changing moment of inertia and cross-section area.}
	\centering \begin{tabular}{CCCCCC}
		\hline
		\text{DOF}&h&\texttt{e}_0(M)&\texttt{r}_0(M)&\texttt{e}_0(V)&\texttt{r}_0(V)\\
		\hline
		42 &0.05    &4.7263e-3 &--    &1.0539e-2  &--   \\
		82 &0.025   &1.1818e-3 &1.99  &2.6378e-3  &1.99 \\
		162&0.0125  &2.9571e-4 &1.99  &6.6143e-4  &1.99 \\
		202&0.0100  &1.8938e-4 &1.99  &4.2426e-4  &1.99 \\
		242&0.00833 &1.3163e-4 &1.99  &2.9545e-4  &1.98 \\
		282&0.00714 &9.6807e-5 &1.99  &2.1781e-4  &1.97 \\
		\hline
		\text{DOF}&h&\texttt{e}_1(M)&\texttt{r}_1(M)&\texttt{e}_1(V)&\texttt{r}_1(V)\\
		\hline
		42 &0.05    &1.4449e-0 &--    &1.4451e-0 &--   \\
		82 &0.025   &7.2251e-1 &0.99  &7.2293e-1 &0.99   \\
		162&0.0125  &3.6130e-1 &0.99  &3.6214e-1 &0.99 \\
		202&0.0100  &2.8907e-1 &0.99  &2.9011e-1 &0.99\\
		242&0.00833 &2.4092e-1 &0.99  &2.4215e-1 &0.99\\
		282&0.00714 &2.0654e-1 &0.99  &2.0795e-1 &0.98 \\
		\hline
	\end{tabular}
	\label{table6:error_M_V_smooth_beam}
\end{table}
\begin{table}[!ht]
	\caption{Error data and rate of convergence of the transverse displacement $w$ and rotation $\beta$ in the beam with smoothly changing moment of inertia and cross-section area.}
	\centering \begin{tabular}{CCCCCC}
		\hline
		\text{DOF}&h&\texttt{e}_0(w)&\texttt{r}_0(w)&\texttt{e}_0(\beta)&\texttt{r}_0(\beta)\\
		\hline
		42 &0.05    &2.8484e-2 &--    &1.8175e-1  &--   \\
		82 &0.025   &1.4218e-2 &1.00  &9.1261e-2  &0.99   \\
		162&0.0125  &7.1060e-3 &1.00  &4.5678e-2  &0.99 \\
		202&0.0100  &5.6845e-3 &1.00  &3.6547e-2  &0.99  \\
		242&0.00833 &4.7370e-3 &1.00  &3.0458e-2  &0.99\\
		282&0.00714 &4.0602e-3 &1.00  &2.6108e-2  &0.99  \\
		\hline
	\end{tabular}
	\label{table6:error_w_beta_smooth_beam}
\end{table}
\begin{figure}[!ht]
	\centering
	\includegraphics[scale=0.45]{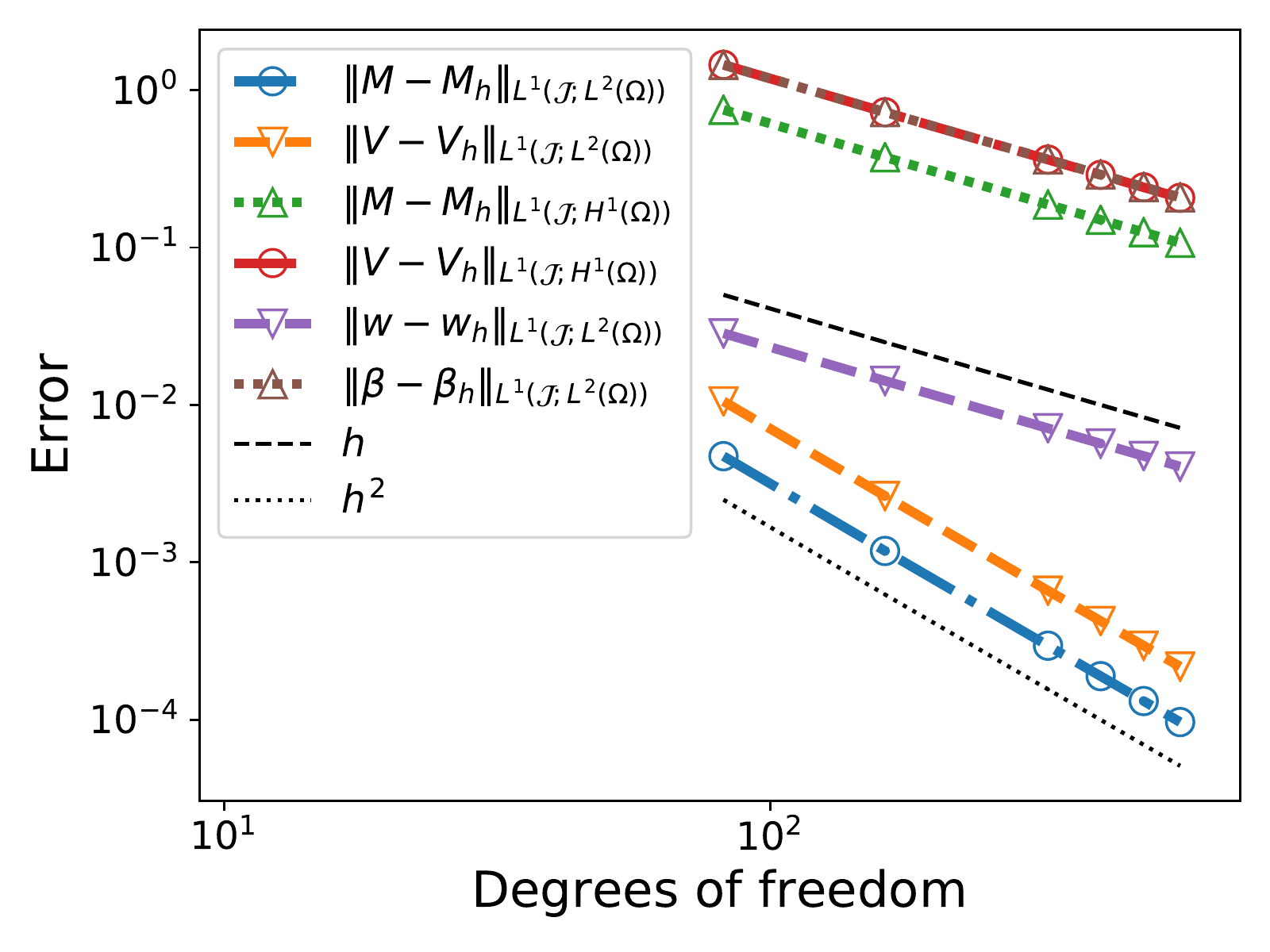}
	
	\caption{Convergence behavior for the clamped beam with smoothly changing moment of inertia and cross-section area when solved with the bending moment formulation. The observation time is  $T=15\,s$ and $\Delta t=0.003$.}
	\label{fig:smooth-beam-errors}
\end{figure}

	


\bibliographystyle{amsplain}
\bibliography{bibliofond}   

\providecommand{\bysame}{\leavevmode\hbox to3em{\hrulefill}\thinspace}
\providecommand{\MR}{\relax\ifhmode\unskip\space\fi MR }
\providecommand{\MRhref}[2]{%
  \href{http://www.ams.org/mathscinet-getitem?mr=#1}{#2}
}
\providecommand{\href}[2]{#2}
\begin{thebibliography}{10}

\bibitem{barbeiro2013laplace}
S{\'\i}lvia Barbeiro, S~Gh Bardeji, and Jos{\'e}~Augusto Ferreira,
  \emph{Laplace transform--finite element method for non fickian flows},
  Computer Methods in Applied Mechanics and Engineering \textbf{261} (2013),
  16--23.

\bibitem{boffi2013mixed}
Daniele Boffi, Franco Brezzi, Michel Fortin, et~al., \emph{Mixed finite element
  methods and applications}, vol.~44, Springer, 2013.

\bibitem{brezzi1974existence}
Franco Brezzi, \emph{On the existence, uniqueness and approximation of
  saddle-point problems arising from lagrangian multipliers}, Publications
  math{\'e}matiques et informatique de Rennes (1974), no.~S4, 1--26.

\bibitem{celiker2006locking}
Fatih Celiker, Bernardo Cockburn, and Henryk~K Stolarski, \emph{Locking-free
  optimal discontinuous galerkin methods for timoshenko beams}, SIAM Journal on
  Numerical Analysis \textbf{44} (2006), no.~6, 2297--2325.

\bibitem{chapelle2010finite}
Dominique Chapelle and Klaus-Jurgen Bathe, \emph{The finite element analysis of
  shells-fundamentals}, Springer Science \& Business Media, 2013.

\bibitem{christensen2012theory}
Richard Christensen, \emph{Theory of viscoelasticity: an introduction},
  Elsevier, 2012.

\bibitem{ewing2002sharp}
Richard~E Ewing, Yanping Lin, Tong Sun, Junping Wang, and Shuhua Zhang,
  \emph{Sharp l 2-error estimates and superconvergence of mixed finite element
  methods for non-fickian flows in porous media}, SIAM journal on numerical
  analysis \textbf{40} (2002), no.~4, 1538--1560.

\bibitem{flugge1975viscoelasticity}
W~Fl{\"u}gge, \emph{Viscoelasticity springer-verlag}, Berlin Google Scholar
  (1975).

\bibitem{golden2013boundary}
John~M Golden and George~AC Graham, \emph{Boundary value problems in linear
  viscoelasticity}, Springer Science \& Business Media, 2013.

\bibitem{gripenberg1990volterra}
Gustaf Gripenberg, Stig-Olof Londen, and Olof Staffans, \emph{Volterra integral
  and functional equations}, vol.~34, Cambridge University Press, 1990.

\bibitem{gutierrez2014engineering}
Danton Gutierrez-Lemini, \emph{Engineering viscoelasticity}, Springer, 2014.

\bibitem{hanks2003fluid}
Richard~W Hanks, \emph{Fluid dynamics (chemical engineering)},  (2003).

\bibitem{hernandez2019}
E~Hernandez, C~Naranjo, and J~Vellojin, \emph{Modelling of thin viscoelastic
  shell structures under reissner--mindlin kinematic assumption}, Applied
  Mathematical Modelling \textbf{79} (2020), 180--199.

\bibitem{karaa2015optimal}
Samir Karaa and Amiya~K Pani, \emph{Optimal error estimates of mixed fems for
  second order hyperbolic integro-differential equations with minimal
  smoothness on initial data}, Journal of Computational and Applied Mathematics
  \textbf{275} (2015), 113--134.

\bibitem{lepe2014locking}
Felipe Lepe, David Mora, and Rodolfo Rodr{\'\i}guez, \emph{Locking-free finite
  element method for a bending moment formulation of timoshenko beams},
  Computers \& Mathematics with Applications \textbf{68} (2014), no.~3,
  118--131.

\bibitem{lovadina2011locking}
Carlo Lovadina, David Mora, and Rodolfo Rodr{\'\i}guez, \emph{A locking-free
  finite element method for the buckling problem of a non-homogeneous
  timoshenko beam}, ESAIM: Mathematical Modelling and Numerical Analysis
  \textbf{45} (2011), no.~4, 603--626.

\bibitem{matei2018mixed}
A~Matei, S~Sitzmann, K~Willner, and BI~Wohlmuth, \emph{A mixed variational
  formulation for a class of contact problems in viscoelasticity}, Applicable
  Analysis \textbf{97} (2018), no.~8, 1340--1356.

\bibitem{miao1998finite}
Chen~Chuan Miao and Shih Tsimin, \emph{Finite element methods for
  integrodifferential equations}, vol.~9, World Scientific, 1998.

\bibitem{mukherjee2003elastic}
S~Mukherjee and Glaucio~H Paulino, \emph{The elastic-viscoelastic
  correspondence principle for functionally graded materials, revisited}, J.
  Appl. Mech. \textbf{70} (2003), no.~3, 359--363.

\bibitem{payette2010nonlinear}
GS~Payette and JN~Reddy, \emph{Nonlinear quasi-static finite element
  formulations for viscoelastic euler-bernoulli and timoshenko beams},
  International Journal for Numerical Methods in Biomedical Engineering
  \textbf{26} (2010), no.~12, 1736--1755.

\bibitem{rognes2010}
Marie Rognes and Ragnar Winther, \emph{Mixed finite element methods for linear
  viscoelasticity using weak symmetry}, Mathematical Models \& Methods in
  Applied Sciences - M3AS \textbf{20} (2010).

\bibitem{saedpanah2016existence}
Fardin Saedpanah, \emph{Existence and convergence of g alerkin approximation
  for second order hyperbolic equations with memory term}, Numerical Methods
  for Partial Differential Equations \textbf{32} (2016), no.~2, 548--563.

\bibitem{shaw2001optimal}
Simon Shaw and John~R Whiteman, \emph{Optimal long-time {$ L_p (0, T)$}
  stability and semidiscrete error estimates for the volterra formulation of
  the linear quasistatic viscoelasticity problem}, Numerische Mathematik
  \textbf{88} (2001), no.~4, 743--770.

\bibitem{sinha2009mixed}
Rajen~K Sinha, Richard~E Ewing, and Raytcho~D Lazarov, \emph{Mixed finite
  element approximations of parabolic integro-differential equations with
  nonsmooth initial data}, SIAM Journal on Numerical Analysis \textbf{47}
  (2009), no.~5, 3269--3292.

\end{thebibliography}

\end{document}